\documentclass{article}

\usepackage{amsmath,amssymb,amsthm,latexsym,verbatim,natbib,yfonts,xcolor,hyperref}
\newtheorem{df}{Def}[section]  
\newtheorem{definition}[df]{Definition}

\newtheorem{thm}[df]{Theorem}

\newtheorem{cor}[df]{Corollary}
\newtheorem{lem}[df]{Lemma}
\newtheorem{example}[df]{Example}

\newcommand{\FU}[2]{{#2}^{#1}}

\author{Fausto Barbero, University of Helsinki}
\title{Some observations about generalized quantifiers in logics of imperfect information}

\begin{document}
\maketitle
%
\begin{abstract}
We analyse the two definitions of generalized quantifiers for logics of dependence and independence that have been proposed by F. Engstr\"om, comparing them with a more general, higher-order definition of team quantifier. We show that Engstr\"om's definitions (and other quantifiers from the literature) can be identified, by means of appropriate lifts, with special classes of team quantifiers. We point out that the new team quantifiers express a quantitative and a qualitative component, while Engstr\"om's quantifiers only range over the latter. We further argue that Engstr\"om's definitions are just embeddings of the first-order generalized quantifiers into team semantics, and fail to capture an adequate notion of team-theoretical generalized quantifier, save for the special cases in which the quantifiers are applied to flat formulas. We also raise several doubts concerning the meaningfulness of the monotone/nonmonotone distinction in this context. In the appendix we develop some proof theory for Engstr\"om's quantifiers.
\end{abstract}

\color{red}{This is the manuscript accepted for publication by the Review of Symbolic Logic, copyright by Cambridge University Press. The version of record can be found at \href{https://doi.org/10.1017/S1755020319000145}{https://doi.org/10.1017/S1755020319000145} . }

\color{blue}{The article has been prepared under the Academy of Finland
project 286991, ``Dependence and Independence in Logic: Foundations and Philosophical
Significance,'' and revised under the Academy of Finland project 316460, ``Semantics of
causal and counterfactual dependence.'' }

\color{black}

    \section{Introduction} \label{GENQUAN}

Languages of imperfect information are a family of logical formalisms which allow the semantical analysis of notions, such as dependence and independence, that cannot be captured by classical first-order logic. Stemming from the partially ordered quantifiers of \cite{Hen1961}, logics have been developed which express in full generality \emph{functional dependence}: Independence-Friendly logic \citep{HinSan1989,ManSanSev2011} and Dependence logic \citep{Vaa2007}. Following similar approaches, newer logical languages were introduced which extend first-order logic with other notions of (in)dependence, such as \emph{database dependencies} \citep{GraVaa2013,Gal2012}, \emph{probabilistic independence} \citep{DurHanKonMeiVir2016}, \emph{quantum probabilities} \citep{HytPaoVaa2012} and \emph{causal dependence} \citep{BarSan2018}. 

Earlier presentations of these logics ground the semantics on intuitions related to Skolemization \citep{San1993} or to semantic games \citep{HinSan1989}; however, the unifying background of many of the recent developments in the field is instead  \emph{team semantics} \citep{Hod97,Hod97b,Vaa2007}. Team semantics is a generalization of the Tarski-style, compositional semantics of first-order logic; according to it, the notion of ``satisfaction of a formula by an assignment'' is replaced by the notion of ``satisfaction of a formula by a team'' (a \emph{set} of assignments). Within this more general framework, it becomes possible to embed in logical formulas notions that cannot be expressed in first-order logic; new concepts can be incorporated in the form of atomic formulas, special quantifiers, or new logical operators. It is in this spirit that \cite{Eng2012} -- followed by \cite{EngKon2013} and \cite{EngKonVaa2013} -- proposed two definitional schemes for the introduction of generalized quantifiers in Dependence logic; one scheme for the (upwards) monotone quantifiers, and a more complex one intended for capturing also the non-monotone case. He also extended his definitions to slashed and backslashed quantifiers, along the lines, respectively, of Independence-Friendly ($IF$) logic and Dependence-Friendly ($DF$) logic. Much in the same spirit, \cite{Sev2014} gave a definition of the ``most'' quantifier in the context of Independence-Friendly ($IF$) logic. 


In the present paper, we will analyze some aspects of this treatment of generalized quantifiers, focusing on extensions of first-order, $IF$ and $DF$ logic. We will also present a more general, higher-order notion of team-theoretical generalized quantifier, and see how the quantifiers of Engstr\"om fit in this more general scheme.

The plan of the paper is as follows:\\
1) We will argue that Engstr\"om's (and Sevenster's) reading of the generalized quantifiers is biased by misleading parallels with first-order semantics. Once a proper, team-semantical, reading of these quantifiers is applied, the need to restrict Engstr\"om's (first) definition to the monotone case disappears. Yet, the class of quantifiers which is captured by this definition in the nonmonotone case   is not the class of (Mostowski's) first-order non-monotone quantifiers. \\
2) We will introduce (section \ref{SECSECOND}) a more general definition of generalized quantifier (\emph{team quantifiers}), illustrate the notion with some examples, compare it with the so-called second-order generalized quantifiers and show (sections \ref{REA} and \ref{SECNONMON}) that Engstr\"om's first and second definition of generalized quantifiers can be seen as special cases of this more general notion.\\
3) We will point out that the second definition of Engstr\"om's (and the first one, when restricted to the monotone case) manages to model correctly the first-order generalized quantifiers, but only when the quantifiers are applied to \emph{flat} formulas.\\
4) We will further defend, with a few arguments and examples, the extension of Engstr\"om's first clause beyond the monotone case. Among other things, a) we show (in a technical appendix) that many good synctactical properties of $IF$ logic are preserved when the new quantifiers are added, 
and we see that monotonicity plays no role in these proofs; b) we investigate (at the end of section \ref{SECENG}) which class of Engstr\"om quantifiers preserves the locality property of logics, concluding that the relevant property in this context is not monotonicity, but rather union closure.

Section \ref{SECNOT} is a glossary of the notations that are used most commonly throughout the paper; section \ref{SEM} reviews the syntax and semantics of $IF$ and $DF$ logics; and section \ref{SECENG} reviews thoroughly Engstr\"om's quantifiers in their most basic version.

\section{General notation} \label{SECNOT} 

We summarize here our use of letters and notations. Most of the definitions are deferred to later sections. The reader can skip this section and refer to it as needed.

Greek letters denote formulas, with the exception of $\sigma$, which is reserved for signatures.

Small letters ($u,v,x,y,z...$) denote variables (standing for individuals), while  $U,V,W$ denote finite sets of such variables. The letters $i$ and $j$ are reserved for indexes.

The letter $M$ denotes a first-order structure, i.e. a pair $(dom(M),I_M)$, where $dom(M)$ is the domain of the structure (a set), while $I_M$ is a function which assigns appropriate interpretations to the elements of the signature of $M$ (constant, relation and function symbols). We abuse notation and write $M$ for $dom(M)$ when there is no risk of ambiguity; in particular, we will often write, for brevity, $\wp(M)$ instead of  $\wp(dom(M))$ (the power set of the domain of $M$). The letters $P,S$ denote either 1) subsets of $dom(M)^n$ for some structure $M$, or 2)(first-order) relation symbols (either in the role of constants or variables); in this second case, we write $P^M,S^M$ as abbreviations for $I_M(P),I_M(S)$. Following a common convention, we sometimes refer to a structure as $(M,S_1,\dots,S_n)$ if we want to emphasize that the domain of the interpretation function of the structure is the set of relation symbols $\{S_1,\dots,S_n\}$.

The letter $Q$, and its variants ($Q', Q_i$...) are used to denote (global) generalized quantifiers (to be reviewed later in the text), possibly including $\forall$ and $\exists$. The letter $R$ is reserved for generalized quantifiers \emph{distinct from $\forall$ and $\exists$}. The same letters will be used to denote quantifier symbols in the object language. Sets of quantifiers are denoted by the italic $\mathcal Q$ or $\mathcal R$. The local quantifiers (relative to a domain $M$) corresponding to $Q$, resp. $R$, are denoted as $Q^M$, resp. $R^M$. The symbols $\hat{Q}$, resp. $\hat{R}$ are reserved for global team quantifiers (to be introduced in section \ref{SECSECOND}) and the corresponding symbols in the object languages; $\hat{Q}^{M,X}$, resp. $\hat{R}^{M,X}$ denote their local versions 
($X$ here denotes a team, to be defined below).  Sets of team quantifiers are denoted by $\hat{\mathcal Q}$ or $\hat{\mathcal R}$.

More generally, throughout the paper we need symbols referring to inviduals and to first, second and third order relations. These four levels are distinguished by using, respectively, small, capital, gothic and capped letters. We summarize these conventions in a table; notice the special treatment of symbols for local (Mostowski) quantifiers.  (The local team quantifiers that we introduce in this paper, instead, do not fit well in this table; see section \ref{SECSECOND} for a discussion of their location in the type hierarchy.)

\vspace{10pt}

\noindent\begin{tabular}{|c|c|c|c|}
\hline
\textbf{Type}             & \textbf{Member of}                & \ \textbf{Variables/constants} \     &  \  \textbf{Local quantifiers} \ \\
\hline
individual       & $dom(M)$                 & $u,v,x,y,z$     &    				\\  
\hline
first-order      & $\wp(dom(M))$            & $P, S$            & 				\\  
\hline
 \ second-order \       & $\wp(\wp(dom(M)))$   & $\textgoth{P}$, $\textgoth{S}$    & $Q^M$, $R^M$				  \\  
\hline
third-order       & \ $\wp(\wp(\wp(dom(M))))$ \ &  $\hat P, \hat S$            & 			\\  
\hline
\end{tabular}

\vspace{10pt}

Notice that what we mean by placing e.g. $P$ in the ``first-order'' row is that $P$ is a symbol which is used to denote a first-order set (a member of $\wp(dom(M))$), or more generally a relation in $\wp(dom(M)^n)$, for some $n$. Somewhat confusingly, $P$ is sometimes called in the literature a ``second-order variable''; we will prefer the locution ``variable for first-order relations''.

The letters $s,s',t$... are reserved for assignments.
An \textbf{assignment} $s$ on a structure $M$ is a function $dom(s)\rightarrow dom(M)$, where $dom(s)$ is any finite set of variables. We will use some operations on assignments and relations among them. Assuming $s$ to be an assignment $dom(s)\rightarrow dom(M)$,
\begin{itemize}
\item For any $a\in dom(M)$ and any variable $v$, we denote as $s(a/v)$ the assignment of domain $dom(s)\cup \{v\}$ given by
\[
(s(a/v))(x) = \left\{\begin{array}{ll}a & \text{ if $x$ is $v$ }     \\
 s(x) & \text{ if $x\in dom(s)\setminus\{v\}$}.
\end{array}
\right.
\]  
\item $s(a_1/v_1,\dots,a_n/v_n)$ stands for $s(a_1/v_1)\dots(a_n/v_n)$.
\item For any set of variables $V\subseteq dom(s)$, we write $s_{\upharpoonright V}$ for the restriction of $s$ to $V$.
\item For any variable $v\in dom(s)$, we write $s_{-v}$ for the restriction of $s$ to $dom(s)\setminus \{v\}$.

\end{itemize}

The letters $X,Y,Z$ 
 are reserved for teams (sets of assignments with a common domain). The letters $F,G,H$ denote functions having a team as domain.

We will use different notations for different notions of extension by generalized quantifier. If $L$ is a language and $\mathcal Q$ a set of quantifiers, $L^{\mathcal Q}$ will denote the usual Mostowski extension of $L$ (reviewed in section \ref{SECENG}); $L(\mathcal Q)$ will denote the extension in the sense of Engstr\"om's first definition of quantifier (also reviewed in section \ref{SECENG}); extensions according to Engstr\"om's second definition will be denoted as $L^b(\mathcal Q)$. $L[\hat{\mathcal Q}]$ will denote an extension by \emph{team} quantifiers, according to the semantics we introduce in section \ref{SECSECOND} If $\mathcal Q = \{Q\}$ (resp. $\hat{\mathcal Q} = \{\hat Q\}$) is a singleton, we simply write $L^{ Q}$, $L(Q)$, $L^b(Q)$ (resp. $L[\hat{Q}]$), omitting the curly braces.

The symbols $[\psi]^{\vec v}_{M,s}$, $[\psi]^{\vec v}_{M,X}$, $[\psi]^{\vec v,V}_{M,X}$, $|\psi|^{\vec v}_{M}$ will denote various notions of ``meaning'' of the formula $\psi$ relative to a few parameters. All these notions are introduced in section \ref{SECSECOND}

\section{Syntax and semantics of $IF$, $IF^*$ and $DF$ logics} \label{SEM} 



We summarize here the syntax and the team semantics of $IF$ logic. A justification for the semantical rules could be given in terms of a game-theoretical semantics; the interested reader may consult \cite{ManSanSev2011}. 

The syntax of $IF$ formulas is similar to that of first-order logic: 
given a signature $\sigma$, we define $\sigma$-terms and atomic formulas as for first-order logic, and the formulas of $IF(\sigma)$ are given by following clauses:
\[
\alpha  \ | \ \neg \alpha \ | \ \psi \land \chi \ | \ \psi \lor \chi \ | \ (\exists v/V)\psi \ | \ (\forall v/V)\psi  
\]
where 
 $v$ is a variable, $V$ a finite set of variables, $\psi$ and $\chi$ are $IF(\sigma)$ formulas, and $\alpha$ is an atomic formula of signature $\sigma$. In the following, we will usually omit reference to the signature. 

The set $V$ occurring in $(\exists v/V)$ or $(\forall v/V)$ is called the \textbf{slash set} of the quantifier. For brevity, we will write $\exists v$ and $\forall v$ for the quantifiers $(\exists v/\emptyset)$ and $(\forall v/\emptyset)$ which have empty slash set. A formula whose quantifiers all have empty slash set will be said to be \textbf{first-order}.

 The set of free variables of a formula $\psi$, $FV(\psi)$, should be defined with care, so as to include also those variables from slash sets that do not fall in the scope of quantification over the same variable. The following inductive clauses do the job:
\begin{itemize}
\item For $\alpha$ atomic $IF$ formula, $FV(\alpha) = FV(\neg\alpha)$ is the set of all the variables that occur in $\alpha$.
\item $FV(\psi\land\chi) = FV(\psi\lor\chi) = FV(\psi)\cup FV(\chi)$.
\item $FV((\exists v/V)\psi) = FV((\forall v/V)\psi) = (FV(\psi)\setminus \{v\}) \cup V$.
\end{itemize}
A variable which is not free in a formula is said to be bound. So, for example, the $IF$ formula $\exists z(\exists y/\{x\})R(y,z)$ has exactly one free variable, $x$, in spite of the fact that $x$ does not occur in the atomic part of the formula; $z$ and $y$ are bound. 

As usual, if $FV(\varphi) = \emptyset$, $\varphi$ is called a sentence.

Now we move to team semantics. Given a finite set of variables $V$ and a structure $M$, a \textbf{team} $X$ of domain $V$ on $M$ is a set of assignments $V \rightarrow dom(M)$. We will denote the domain $V$ as $dom(X)$.  We can lift many notational conventions (as they were fixed in section \ref{SECNOT}) from the level of assignments to teams:
\begin{itemize}
\item  For any set of variables $V\subseteq dom(X)$, we write $X_{\upharpoonright V} := \{s_{\upharpoonright V} \ | \ s\in X \}$.
\item For any variable $v\in dom(X)$, we write $X_{-v}:= \{s_{-v} \ | \ s\in X \}$.
\item Given a team $X$ and a sequence $v_1,\dots,v_n$ of variables in $dom(X)$, we can define an associated $n$-ary relation $X(v_1,\dots,v_n) := \{(s(v_1),\dots,s(v_n)) \ | \ s\in X\}$.
\end{itemize}

The following operators on teams correspond, respectively, to universal and existential quantification:

\begin{itemize}
\item Given a team $X$ on a structure $M$, and a variable $v$, the \textbf{duplicated team} $X[M/v]$ is defined as the team $\{s(a/v) \ | \ s\in X,a\in dom(M)\}$.
\item Given a team $X$ over a structure $M$, a variable $v$ and a function $F: X \rightarrow \wp(M)$, the \textbf{supplemented team} $X[F/v]$ is defined as the team $\{s(a/v) \ | \ s\in X,a\in F(s)\}$. 
More generally, for a sequence of variables $\vec v = v_1,\dots v_n$ and $F: X \rightarrow \wp(M^n)$, we can define $X[F/\vec v] = \{s(a_1,\dots,a_n / v_1,\dots,v_n) \ | \ s\in X, (a_1,\dots,a_n)\in F(s) \}$.
\end{itemize}

Finally, we need some machinery in order to take care of statements of independence among quantifiers.

\begin{itemize}
\item Given two assignments $s,s'$ with the same domain, and a set of variables $V$, we say that $s$ and $s'$ are \textbf{$V$-equivalent}, and we write $s\sim_V s'$, if $s(x) = s'(x)$ for all variables $x\in dom(s)\setminus V$.
\item Given a team $X$, a structure $M$ and a set $V$ of variables, a function $F:X\rightarrow \wp(M)$ is \textbf{$V$-uniform} if, for all $s,s'\in X$, $s\sim_V s'$ implies $F(s)=F(s')$.   
\end{itemize}

We can now define the team semantics of $IF$ logic, which can be thought of as a ternary relation $M,X\models \varphi$ between a structure $M$, a team $X$ and an $IF$ formula $\varphi$ (``in $M$, $\varphi$ is satisfied by the team $X$''). We assume that the reader is familiar with the Tarskian notion of satisfaction (of a quantifier-free formula) by an assignment, symbolized here as $M,s\models \varphi$. We shall say that a team is \textbf{suitable} for an $IF$ formula $\psi$ provided that $dom(X)\supseteq FV(\psi)$. 



\begin{definition}
We say that a suitable team $X$ satisfies  an  $IF$ formula $\psi$ over a structure $M$, and we write $M,X\models \psi$, according to the following clauses:
\begin{itemize}
\item $M,X\models R(t_1,\dots,t_n)$ if $M,s\models R(t_1,\dots,t_n)$ in the Tarskian sense for every $s\in X$.
\item  $M,X\models \neg R(t_1,\dots,t_n)$ if $M,s\models \neg R(t_1,\dots,t_n)$ in the Tarskian sense for every $s\in X$.
\item $M,X\models t_1 = t_2$ if $M,s\models t_1 = t_2$ in the Tarskian sense for every $s\in X$.
\item $M,X\models t_1 \neq t_2$ if $M,s\models t_1 \neq t_2$ in the Tarskian sense for every $s\in X$.
\item $M,X\models \chi_1\land\chi_2$ if $M,X\models\chi_1$ and $M,X\models\chi_2$.  
\item $M,X\models \chi_1\lor\chi_2$ if there are $Y,Z\subseteq X$ such that $Y\cup Z=X$, $M,Y\models \chi_1$, and $M,Z\models \chi_2$.
\item $M,X\models (\forall v/V)\chi$ if $M,X[M/v]\models \chi$.
\item $M,X\models (\exists v/V)\chi$ if $M,X[F/v]\models \chi$ for some $V$-uniform function $F:X\rightarrow \wp(M)\setminus \{\emptyset\}$.\footnote{The clause given here for existential quantification is not the most commonly used in the literature on $IF$ logic. Typical presentations, such as \cite{ManSanSev2011}, include a so-called \emph{strict} condition, where the function $F$ is a function $X\rightarrow dom(M)$, picking elements rather than sets. \cite{Eng2012} introduced this \emph{lax} version of the clause for existential quantifiers, which we follow here, and which turns out to be more appropriate for the inclusion of generalized quantifiers. For basic $IF$ logic, the strict and lax conditions give rise to equivalent semantics. More generally, the strict and lax clause are interchangeable for any downward closed logic; downward closure is the property expressed in Theorem \ref{DWCR} below.}

\end{itemize}
\end{definition}

Then one can say that a sentence $\varphi$ is \textbf{true in $M$} ($M\models\varphi$) if $M,\{\emptyset\} \models \varphi$.
 Two $IF$ formulas $\psi,\chi$ are \textbf{truth-equivalent} if, for every structure $M$ and every team $X$ on $M$, $M,X\models\psi$ if and only if $M,X\models \chi$.

 It is to be remarked that, if we omit the requirement of $V$-uniformity from the existential case, the above clauses define a semantics for first-order ($FO$) logic, which turns out to be equivalent, at the level of sentences, with the usual Tarskian semantics.\footnote{For singleton teams, the equivalence extends to all $FO$ formulas, in the sense that $M,\{s\}\models \psi$ iff $M,s\models \psi$, for every structure $M$.}

Some of the arguments in the paper will be more easily stated in terms of Dependence-Friendly ($DF$) logic; this language has, instead of slashed quantifiers, backslashed ones; the quantifier $(\exists v\backslash V)$ is interpreted as ``there exists a $v$ dependent only on $V$''. The definition of sets of free variables is analogous to the case of $IF$ logic. The semantical clauses for backslashed quantifiers are:
\begin{itemize}
\item $M,X\models (\forall v\backslash V)\chi$ if $M,X[M/v]\models \chi$.
\item $M,X\models (\exists v\backslash V)\chi$ if $M,X[F/v]\models \chi$ for some $(dom(X) \setminus V)$-uniform function $F:X\rightarrow \wp(M)\setminus \{\emptyset\}$.
\end{itemize}
Via an appropriate embedding\footnote{The idea, which is somewhat more awkward than in the case of $IF$ logic, is to identify a backslashed quantifier $(Qv/V)$ with the first-order $Qv$ in case $(Qv/V)$ only occurs in the scope of quantifications of variables from $V$.}, also $DF$ logic can be seen as a conservative extension of $FO$. We are interested in this language because (differently from $IF$) it shares with $FO$ a property called \emph{locality}, which will be discussed at the end of the next section.

Finally, we must recall that, according to team semantics, a formula $\psi$ (of $IF$, $DF$, etc.) is said to be \textbf{flat} if it satisfies the following condition, for all suitable structures $M$ and teams $X$:
\[
M,X\models \psi \Leftrightarrow \text{(for all $s\in X$, $M,\{s\}\models \psi$)}.
\]
All first-order formulas are flat; so, flat formulas may be thought of as formulas that behave similarly to first-order ones. 


\section{Engstr\"om quantifiers} \label{SECENG}

We now present the generalized quantifiers for $IF$ logic that were introduced in \cite{Eng2012}. 
Given that Engström's definition draws  on earlier work by \cite{Mos1957} and \cite{Lin1966}, we shall start by shortly reviewing their definition. 
 Following Mostowski and Lindstr\"om we can call \emph{global quantifier of type $(n)$} any class $Q$ of structures of the form $(M,S)$, with $S$ a $n$-ary relation symbol. More general quantifiers (of type $(n_1,...,n_k)$) are considered in the literature, but we shall not concern us with them here.

Given a global quantifier $Q$ of type $(n)$, one can associate to each \emph{domain}\footnote{Remember our convention of writing $M$ either for structures or their domains, according to convenience.} $M$ its local quantifier $Q^M := \{S \ | \ (M,S)\in Q\}$; notice that this is a set of $k$-ary relations, i.e. an element of $\wp(\wp(M^n))$. 

Often some restrictions are imposed on the classes of structures which can be considered global quantifiers; the most important of these being \emph{closure under isomorphisms}. Two properties that will be of some importance in our paper are: 
\begin{itemize}
\item (Upwards) monotonicity: for all $M$, if $A\in Q^M$ and $A\subseteq B$, then $B\in Q^M$.
\item Union closure: for all $M$, if $\emptyset \neq \textgoth A \subseteq Q^M$, then $\bigcup \textgoth A \in Q^M$.
\end{itemize}
A monotone quantifier is also union closed, but the converse may fail; that is the case, for example, if $Q^M$ is any singleton different from $\{M\}$.

These generalized quantifiers can be ``added''  to first-order logic by allowing synctactical expressions of the form $Qv_1...v_n\psi$; in the style of Mostowski, these are interpreted by the semantical clause:

\begin{itemize}
\item $M,s\models Qv_1...v_n\psi \text{ if } \{(a_1,\dots,a_n) \ | \ M,s(a_1/v_1,\dots,a_n/v_n)\models \psi\}\in Q^M$. \footnote{More precisely, in this context we should write $Q^{dom(M)}$.}\end{itemize}

When $\mathcal Q$ is a set of generalized quantifiers, we will let $FO^{\mathcal Q}$ denote the extension of first-order logic with the quantifiers in $\mathcal Q$ (interpreted by the clause above). 

Notably, the type $(1)$ local quantifiers $\forall^M = \{M\}$ and $\exists^M = \wp(M)\setminus\{\emptyset\}$, combined with this semantical clause, provide alternative, equivalent definitions of the universal and existential quantifiers. 

We now return to Engström's notion of generalized quantifier. It combines the idea of a local quantifier described above with the clause for the existential quantifier in team semantics described in Definition 3.2. For quantifiers of type $(1)$, this amounts to the following clause: 


\begin{itemize} 
\item $M,X\models (Qv/V)\psi$ if there is a $V$-uniform function $F:X\rightarrow Q^M$ s. t. $M,X[F/v]\models \psi$.
\end{itemize}
We remark that the clause might be easily extended to type $(n)$ generalized quantifiers of the form $(Qv_1\dots v_n/V_1,\dots,V_n)$, but the type $(1)$ quantifiers will mostly suffice for the arguments of this paper.
Notice that by applying this clause to the local quantifiers $\forall^M$ and $\exists^M$ one obtains precisely the team semantics clauses for the universal and existential quantifiers.


Let $\mathcal Q$ be a set of quantifiers. We consider extensions $FO(\mathcal Q)$ and $IF(\mathcal Q)$  of, respectively, $FO$ and $IF$, in which the quantifiers $\mathcal Q$ are interpreted by Engstr\"om's semantical clause. 

\begin{example}
Consider the local quantifier $\exists_{\geq 3}^M := \{S\subseteq dom(M) \ | \ card(S) \geq 3\}$. Let us evaluate the sentence $\forall x \exists_{\geq 3} yP(x,y)$ on a structure $M$: $M\models \forall x \exists_{\geq 3} yP(x,y)$  iff $M,\{\emptyset\}\models \forall x \exists_{\geq 3} yP(x,y)$ iff $M, \{\emptyset\}[M/x]\models \exists_{\geq 3} yP(x,y)$ iff there is a function $F: \{\emptyset\}[M/x] \rightarrow \exists_{\geq 3}^M$ such that $M,\{ \emptyset\}[M/x][F/y]\models P(x,y)$ (we can also write $\{ \emptyset\}[MF/xy]$ for  $\{ \emptyset\}[M/x][F/y]$). This last condition amounts to saying, as in the case of the corresponding Mostowski quantifier\footnote{We will see that this correspondence fails if a quantifier like $\exists_{\geq 3}$ is applied to a more complex, non-flat formula.}, that for each possible value of $x$ there are at least $3$ values for $y$ which stand in the relation $P$ with this value of $x$. If we write $\forall x (\exists_{\geq 3} y/\{x\})P(x,y)$, we are adding the further requirement that we can pick the same $3$ values of $y$ for each value of $x$. Of course, the same can be achieved using the $FO(\exists_{\geq 3})$ sentence $\exists_{\geq 3}y\forall x  P(x,y)$.


But $IF(\mathcal Q)$ sentences cannot always be reduced to $FO(\mathcal Q)$ by means of quantifier swapping or other tricks; for example, in \cite{Sev2014} it is shown that there are sentences of the form $most x(most y/\{x\})\psi$ (where $most^M = \{S\subseteq M \ | \ card(S) \geq card(M)/2 \}$ for $M$ finite structure) which have no equivalent in $FO(most)$.  

A different kind of example is given by the quantifier $\exists_{= 3}^M := \{S\subseteq dom(M) \ | \ card(S) = 3\}$. Here it is immediately evident that the interpretation given by the clause of Engstr\"om is not the Mostowskian one (``there are exactly three''). For example, the sentence $\exists_{=3}xP(x)$ is true on a structure $M$ iff there are \emph{at least} three elements that satisfy $P$; indeed, the existence of an appropriate supplementing function which picks $n\geq 3$ elements entails  the existence of an appropriate function which picks exactly three elements. In this context, it seems to us to be more reasonable to read this quantifier as ``there is a 3-element set'', or, more generally, ``there is a supplementing function which picks $3$-element sets''.
\end{example}

 Contrarily to what we stated above for $IF$ logic, the restriction of the team semantics of $IF(\mathcal  Q)$ to its first-order fragment $FO(\mathcal Q)$ is not equivalent to Tarskian semantics plus the Mostowski-Lindstr\"om definition of the generalized quantifiers. This kind of conservativity may fail for the nonmonotone quantifiers; this claim will be substantiated in the next sections. It is probably this the reason why Engstr\"om required that his definition be applied only to monotone quantifiers. However, throughout the paper we will illustrate that many good properties of Engstr\"om's semantical clause hold as well if the definition is applied to non-monotone quantifiers (see the rest of the section, but also in particular the Appendix), and that there are other reasons that make the restriction to the monotone case not too well justified.

We see that the $IF(\mathcal Q)$ languages all satisfy two fundamental properties of $IF$ logic.

\begin{thm}[Downward closure] \label{DWCR}
Given a formula $\psi$ of $IF(\mathcal  Q)$ 
 such that $M,X\models\psi$, for all $Y\subseteq X$ it holds that $M,Y\models\psi$.
\end{thm}

\begin{proof}
We give this proof for completeness, since the analogous result of \cite{Eng2012} is stated to hold only under the assumption of monotonicity of the quantifiers.  
We prove the statement by induction on the structure of $\varphi$. For $IF$ operators, see the analogous proof for Dependence logic in \cite{Vaa2007}, 5.1. The only new case is $\varphi = (Qv/V)\psi$, for $Q\in \mathcal Q$. Suppose $M,X\models (Qv/V)\psi$. Then there is a $V$-uniform function $F:X\rightarrow Q^M$ such that $M,X[F/v]\models \psi$. Let $Z = \{s\in X[F/v] \ | \ s_{-v}\in Y\}$.  
By induction hypothesis, $M,Z\models \psi$. Let $G$ be the restriction of $F$ to $Y$; clearly it is a $V$-uniform function $Y\rightarrow Q^M$. Then $Z = Y[G/v]$. So, $M,Y\models (Qv/V)\psi$.  All steps are correct also in the limit case that $X[F/v] =\emptyset$.
\end{proof}

\begin{thm}[Empty team property] \label{EMPTYSETR}
Every formula $\psi$ of $IF(\mathcal  Q)$ 
 is satisfied by the empty team, on any structure ($M,\emptyset\models\psi$).
\end{thm}

\begin{proof}
By induction on the syntax of $\psi$.
\end{proof}

It is straightforward to show that also the languages $DF(\mathcal Q)$ have the downward closure and the empty team property. But they have an additional property, locality (i.e. context independence), which is described in the next theorem.

\begin{thm}[Locality]  \label{TEOLOCDF}
Let $\varphi$ be a $DF(\mathcal Q)$ formula, and $X$ a team suitable for $\varphi$. Then, for all structures $M$,
\[
M,X\models\varphi \iff M,X_{\upharpoonright FV(\varphi)}\models \varphi
\]
\end{thm}
\begin{proof}
%
The proof is by induction on the syntax of $\varphi$; it will be easier to prove the seemingly stronger claim that, for each $U\supseteq FV(\varphi)$, 
\[
M,X\models\varphi \iff M,X_{\upharpoonright U}\models\varphi.
\]
  The cases for (possibly negated) atoms, $\land$ and $\lor$ are adequately treated in \cite{Vaa2007}, proof of Lemma 3.27. We consider the quantifier case, $\varphi = (Qv\backslash V)\psi$. \footnote{This case takes care also of the existential quantifier, thanks to our initial choice of the lax semantics.} Let $W := U\cup \{v\}$.

$\Rightarrow$) Assume $M,X\models (Qv\backslash V)\psi$. Then there is a $(dom(X)\setminus V)$-uniform function $F:X\rightarrow Q^M$ such that $M,X[F/v]\models \psi$. By the inductive hypothesis, $M,X[F/v]_{\upharpoonright W}\models \psi$. Now define $F^*:X_{\upharpoonright U}\rightarrow Q^M$ as follows. For each $s\in X_{\upharpoonright U}$, pick an $s'\in X$ such that $s=s'_{\upharpoonright U}$. Then let $F^*(s):= F(s')$; clearly $F^*$ is $(dom(X)\setminus V)$-uniform. Since $dom(X_{\upharpoonright U})\setminus V\subseteq dom(X)\setminus V$, $F^*$ is also $(dom(X_{\upharpoonright U})\setminus V)$-uniform. It can be checked that $X_{\upharpoonright U}[F^*/v] \subseteq X[F/v]_{\upharpoonright W}$. So, by downward closure, $M,X_{\upharpoonright U}[F^*/v]\models \psi$. Therefore $M,X_{\upharpoonright U} \models (Qv\backslash V)\psi$.

$\Leftarrow$) Assume $M,X_{\upharpoonright U}\models (Qv\backslash V)\psi$. Then there is a $(dom(X_{\upharpoonright U})\setminus V)$-uniform function $G: X_{\upharpoonright U}\rightarrow Q^M$ such that $M,X_{\upharpoonright U}[G/v]\models \psi$. Define $G^*:X\rightarrow Q^M$ by the clause $G^*(s):= G(s_{\upharpoonright U})$; $G^*$ is $(dom(X_{\upharpoonright U})\setminus V)$-uniform. It can then be checked that $X[G^*/v]_{\upharpoonright W} = X_{\upharpoonright U}[G/v]$; so $M, X[G^*/v]_{\upharpoonright W}\models \psi$. By the inductive hypothesis, $M, X[G^*/v]\models \psi$. So $M,X\models (Qv\backslash V)\psi$.  
\end{proof}

In \cite{EngKon2013} the same result is stated (without proof) for Dependence logic; but the authors restrict its validity to extensions by monotone quantifiers. The proof above shows that this restriction is unnecessary. Notice however that downward closure was needed in the proof. Suppose we try to extend with Engstr\"om quantifiers some language which is local, but not downward closed\footnote{For example, languages which contain  the independence atoms of \cite{GraVaa2013}, or the inclusion atoms of \cite{Gal2013}.}; we might then need some restrictive assumptions on the additional quantifiers. A different argument for proving locality, which was given in \cite{Gal2013} Theorem 4.22, shows that one such sufficient condition is requiring \emph{union closure} of the quantifiers; i.e., a restriction which is weaker than monotonicity.

\section{Generalized team quantifiers} \label{SECSECOND}

We want to argue that the methodology for generalized quantifier extensions proposed by Engstr\"om is not the most adequate for $IF$,  $DF$ and similar logics, because it does not take into account the second-order nature of the semantics of these logics. One way of supporting this claim is making the following rough observation: on one side, the ``meaning'' (in an extensional sense) of a first-order formula $\psi$ of $n$ free variables can be identified with the set of all assignments of domain $dom(s)$ which satisfy $\psi$.\footnote{In theory, we should consider assignments $s$ such that $dom(s)\supseteq FV(\psi)$, since also these kinds of assignments can satisfy the formula; but thanks to the locality property of $FO$, the assignments with domain $FV(\psi)$ encode all the significant information.}  Once an alphabetical order of the free variables of $\psi$ is fixed, say $\vec v = v_1\dots v_n$, each such assignment can be identified with a corresponding tuple $s(\vec v) = (s(v_1),\dots,s(v_n))$. So, the ``meaning'' of a $FO$ formula of $n$ free variables is a set of $n$-tuples, i.e. an \emph{$n$-ary relation}. On the other hand, formulas of team based logics are satisfied by teams. For simplicity, let us focus on a local logic, such as $DF$. The ``meaning'' of a $DF$ formula $\psi$ with $n$ free variables is the set of all the teams of domain $FV(\psi)$. Assuming again that the variables of $\psi$ are enumerated as $\vec v$, each team $X$ in the ``meaning'' of $\psi$ can be identified with the relation $X(\vec v):= \{(s(v_1),\dots,s(v_n)) \ | \ s\in X\}$. So, the ``meaning'' of a $DF$ formula is a \emph{set} of $n$-ary relations. 
 If variables in team logics refer to sets and relations, it seems then unreasonable that quantifiers of team logics be not shifted to a higher order. The typical first-order quantifiers assert that a formula holds for a \emph{quantity of elements}. We should expect, then, that the typical generalized quantifiers for team logics state that a formula holds for a certain \emph{quantity of sets}; and perhaps, as a second dimension, it could express the \emph{size of the sets}. Thus, the interpretation $Q^M$ of a quantifier $Q$ in a structure $M$ should not be a subset of $\wp(M)$, but a subset of $\wp(\wp(M))$.
In an extension of $DF$ or $IF$ logic, it would seem legitimate to say things like ``$\psi$ holds for most six-element sets'', ``there is a set of size smaller than three for which $\psi$ holds'', ``$\psi$ holds on all cofinite sets'', and so on. Quantifiers of this kind appear in the literature under the name of \emph{second-order generalized quantifiers}.\footnote{Second-order generalized quantifiers were introduced in two slightly different forms in \cite{BurVol1998} and \cite{And2002}, respectively; significant results on their definability theory are found e.g. in \cite{Kon2010}.}

Actually, in the context of team semantics, this is still an oversimplification; the quantifiers we are after will behave as second-order generalized quantifier only when they occur sentence-initially. We will be guided to a reasonable definition by analogy with the first-order case.
 As we have seen, a local quantifier $Q^M$ in the sense of Mostowski is a second-order object -- a set of $n$-ary relations, if its type is $(n)$. In order to assess whether a $FO^Q$ formula $\varphi = Qv_1\dots v_n\psi$ is satisfied on $M$ by a given assignment $s$ (of domain $dom(s)\supseteq FV(\varphi)$), one needs to take into consideration an object that we may call the \emph{meaning of $\varphi$ relative to $M,s$ and $\vec v$} ($=v_1,\dots,v_n$): $[\varphi]_{M,s}^{\vec{v}}:= \{(a_1,\dots,a_n) \ | \ M,s(a_1/v_1,\dots,a_n/v_n)\models \psi\}$. If $[\varphi]^{\vec v}_{M,s}\in Q^M$, then $\varphi$ is satisfied by $s$ in $M$. This ``meaning'' encodes all the ways in which $s$ can be extended so as to satisfy $\psi$. Notice that $[\varphi]^{\vec v}_{M,s}$ can be an element of $Q^M$ because it is a \emph{first-order} object, i.e., an ($n$-ary) relation.


With this in mind, it seems reasonable to introduce a notion of \emph{meaning of $\varphi$ relative to $M,X$ and $\vec v$}, where $X$ is a \emph{team} such that $dom(X)\supseteq FV(\varphi)$. 
In analogy with the first-order case, such a ``meaning'' should encode all the possible ways in which the team $X$ can be extended over variable $v$ so as to satisfy formula $\psi$. 
But it seems hard to identify a natural notion of ``extending a team by means of a set''; this is the reason why our local quantifiers will not be the usual second-order quantifiers. Instead, it is reasonable to extend a team $X$ by extending \emph{each assignment of the team} with a set (supplementing). This amounts to picking a function $F:X\rightarrow \wp(dom(M))$. Thus, the ``meaning'' of $\psi$ relative $M,X$ and $v$ will be a set of supplementing functions:
\begin{equation} \label{GREATDEF}
[\varphi]^{v}_{M,X}=\{ F: X\rightarrow \wp(M) \ \ | \ \ M,X[F/v] \models \psi \}.
\end{equation}
This object is an element of $\wp(\wp(M)^X)\subset\wp(\wp(X\times\wp(M)))$, i.e., if we ignore the contribution made by the team, a third-order object. However, the role it will play is mostly affine to a second-order notion of meaning; for this reason, we abuse our notation and use gothic variables such as \textgoth{F} for these kinds of objects.
In the case of $IF$ logic, we will also need to take care of the fact that each quantifier can also have a slash set $V$; the correct semantical object to consider in this case seems then to be:
\begin{equation} \label{GREATDEF2}
[\varphi]^{v,V}_{M,X}= \{ F: X\rightarrow \wp(M) \ \ | \ \ F \text{ is } V\text{-uniform and } M,X[F/v] \models \psi \}.
\end{equation}
``Meanings'' for $DF$ logic formulas can be devised analogously. 

Our \emph{local team quantifiers} will then have to be sets of such meanings; that is, elements of $\wp(\wp(\wp(M)^X))$. Inoring the role of $X$, these are fourth-order objects, but, again, the part they play is analogous to that of local second-order quantifiers (which are \emph{third} order objects). So, abusing again notation, we will denote them as $\hat Q^{M,X}$. Notice that we have a second parameter for the localization, i.e. a team $X$; this is not a strictly necessary step, but it will help us work, in the following, with simpler notations. A (\emph{global}) team quantifier $\hat Q$ might be thought of as a class of local quantifiers indexed by a pair $(M,X)$ ranging over structure-team pairs (with the obvious restriction that $X$ be a team \emph{on $M$}).

Given a set $\hat{\mathcal Q}$ of team quantifiers, we denote by $IF[\hat{\mathcal Q}]$  the language $IF$ extended with the syntactical clause that states that $(\hat Qv/V)\psi$ is an $IF[\hat{\mathcal Q}]$ formula, provided that $\hat Q\in \hat{\mathcal Q}$, $\psi$ is an $IF[\hat{\mathcal Q}]$ formula, $v$ is a variable, and $V$ a finite set of variables. The corresponding semantical clause, for each $\hat Q\in \hat{\mathcal Q}$, is:
\begin{itemize}
\item $M,X\models (\hat Qv/V)\psi$ if $[\psi]^{v,V}_{M,X} \in \hat Q^{M,X}$. 
\end{itemize}

The treatment of extensions of $DF$ logic is similar. Extending the definition to $n$-ary quantifiers is only a matter of greater notational complexity; unary quantifiers will suffice for our arguments in the paper.

Notice that, for what regards the special case in which a team quantifier $\hat Qv$ occurs sentence-initially (say, in front of a formula $\psi$ s.t. $FV(\psi) = \{v\}$), there is a bijection between the set $\wp(M)^{\{\emptyset\}}$ of functions $F:\{\emptyset\} \rightarrow \wp(M)$ and the set of sets of the form $F(\emptyset)$ (for some $F\in \wp(M)^{\{\emptyset\}}$). Therefore, the ``meaning'' that is used to check the truth of  $\hat Qv\psi$ over a structure $M$, that is, the set $[\psi]^{v}_{M,\{\emptyset\}}=\{ F: \{\emptyset\}\rightarrow \wp(M) \ \ | \ \ M,X[F/v] \models \psi \}$, can be identified, through the bijection, with a second-order object, a subset of $\wp(M)$:
\[
   |\psi|^{v}_{M} := \{F(\emptyset) \ | \ F: \{\emptyset\}\rightarrow \wp(M) \text{ and }   M,X[F/v]\models \varphi\}.
\]

 Thus, when evaluating sentence-initial quantifiers, our earlier intuitive statement that the meaning of the subformula $\psi$ is a family of subsets of $dom(M)$ is somewhat justified. Similarly, in the sentence-initial case the local quantifier $\hat Q^{M,\{\emptyset\}}$ can be identified with a third-order object (an element of $\wp(\wp(\wp(M)))$):
\[
|\hat Q|^M := \{ \{F(\emptyset) \ | \ F\in\textgoth{F}\} \ | \ \textgoth{F}\in \hat Q^{M,\{\emptyset\}}\};
\]

\noindent this is the reason for our choice of sticking to the third-order notation $\hat Q$. It is then straightforward to see that the semantical clause can be restated, for sentence-initial quantifiers and the singleton team $\{\emptyset\}$, as 
\begin{itemize}
\item $M,\{\emptyset\}\models \hat Qv\psi$ if $|\psi|^{v}_{M} \in |\hat Q|^M$. 
\end{itemize}

We illustrate the notion of team quantifier (and compare it to previous approaches) with some examples. But first we wish to point out that also for team quantifiers there is a notion of (upwards) monotonicity. A quantifier $\hat Q$ is \emph{monotone (in the team-theoretical sense)} if, for every domain $M$ and every $\textgoth{F},\textgoth{F}'\in \wp(\wp(\wp(M)^X))$ (sets of supplementing functions),  $\textgoth{F}\in\hat Q^M$ plus $\textgoth{F}\subseteq \textgoth{F}'$ implies $\textgoth{F}'\in\hat Q^M$. 

\begin{example}
Consider the following local team quantifiers:
\[
\hat \exists^{M,X} := \{\textgoth{F}\in \wp(\wp(M)^X) \ | \  \exists F\in \textgoth{F}\cap (\wp(M)\setminus \{\emptyset\})^X \}
\]
\[
\hat \forall^{M,X} := \{\textgoth{F}\in \wp(\wp(M)^X) \ | \  \exists F\in \textgoth{F}\forall s\in X(F(s)=dom(M))\}
\]
It is straightforward  to see that these are nothing else than the usual quantifiers $\exists,\forall$ of team semantics, in the sense that, for any formula $\psi$ of any language considered in this paper, any structure $M$ and any team $X$, and any variable $v$,
\[
M,X\models \hat \exists v  \psi \iff M,X\models \exists v\psi
\]
\[
M,X\models \hat \forall v \psi \iff M,X\models \forall v\psi
\]
(similar statements hold for slashed and backslashed versions of the quantifiers).

However, notice that  $\hat\exists$ and $\hat\forall$ are both upwards monotone in the team-theoretical sense. We might say that, from a higher-order perspective, they are both existential quantifiers (stating the existence of a certain function).
\end{example}

\begin{example} \label{EXTHREE}
Let us consider the quantifier ``there are exactly three''. There are at least four distinct ways to represent this kind of quantification by means of a team quantifier. The first of these is described by the local quantifier
\[
\hat\exists_{=3}^{M,X} := \{ \textgoth{F}\in \wp(\wp(M)^X) \ | \ \exists F\in \textgoth{F}\forall s\in X (card(F(s)) = 3)\}.  
\]
If the ``meaning'' of a formula $\psi$ is in $\hat\exists_{=3}^{M,X}$, that means that ``there is at least one function $F$ which picks 3-elements sets'', such that $M, X[F/v]\models \psi$. This is equivalent to the Engstr\"om quantifier $\exists_{=3}$; but it should be clear that, in the context of team quantifiers, this is a perfectly reasonable quantifier, contrarily to what was claimed by Engstr\"om (we will return more systematically on this point in section \ref{REA}). Notice that, in the team-theoretical sense, this is an (upwards) monotone quantifier. 

The second possible interpretation of ``there are exactly three'' is the following:
\[
\hat\exists_{=3,nm}^{M,X} := \{ \textgoth{F}\in \wp(\wp(M)^X) \ | \ \textgoth{F}\neq\emptyset \land \forall F\in \textgoth{F}\forall s\in X(card(F(s)) = 3)\}.  
\]
Now $M,X\models\hat\exists_{=3}\psi$ states that 1) there is at least one function $F$ picking $3$-element sets, such that $M,X[F/v]\models\psi$; and 2), that any supplementing function $F$ which satisfies the condition $M,X[F/v]\models\psi$ is a function which picks $3$-element sets. We will see in section \ref{SECNONMON} how this quantifier relates to Engstr\"om's \emph{second} definition of generalized quantifier, which is instead instantiated by the following quantifier:
\[
\hat\exists_{=3,b}^{M,X} := \{ \textgoth{F}\in \wp(\wp(M)^X) \ | \ 
 \exists F\in \textgoth{F}[\forall s\in X (card(F(s)) = 3) \land 
\]
\[ 
 \land\forall F'\geq F\forall s\in X(card(F(s)) = 3)]\}.  
\]
where $F'\geq F$ means that $F'(s)\supseteq F(s)$ for all $s\in X$.

The fourth interpretation of ``there are exactly three'' is the following:
\[
\hat\exists_{=\hat 3}^{M,X} := \{ \textgoth{F}\in \wp(\wp(M)^X) \ | \ card(\textgoth{F}) = 3 \text{ and } \forall F\in \textgoth{F}(F\neq \emptyset) \}.
\]
$M,X\models \hat\exists_{=\hat 3}\psi$ states the existence of exactly three supplementing functions $F_1,F_2,F_3$ such that, for $i=1..3$, $M,X[F_i/v]\models\psi$. The requirement of nonemptiness seems to be necessary in order to avoid triviality, since the most typical logics based on team semantics all have the empty team property. This third definition is close in spirit to the Most quantifier that was introduced in \cite{DurEbbKonVol2011}, and which we review in the next example.  
\end{example}

\begin{example}
There has been already at least one proposal of a properly higher-order, team-theoretical quantifier. \cite{DurEbbKonVol2011} studied a ``Most'' quantifier defined by the following semantical clause (over \emph{finite} structures):
\[
M,X\models \operatorname{Most} v\psi \Leftrightarrow 
\]
\[
\Leftrightarrow \text{ there are at least $card(M^X)/2$ functions } X\rightarrow M \text{ s.t. } M,X[F/v]\models \psi.
\] 
This can be straightforwardly expressed as the semantics of a team quantifier. Let $\wp^1(M)$ denote the set of singleton subsets of $dom(M)$. Then $M,X\models \operatorname{Most} v\psi$ if and only if $[\psi]^{v}_{M,X} \in Most^{M,X}$, where 
\[
Most^{M,X}:=\{\textgoth{F}\in \wp(\wp(M)^X) \ | \ card(\textgoth{F}\cap \wp^1(M)^X)\geq card(M^X)/2)\}.
\]
(Notice the slight difference: the functions in the definition of \cite{DurEbbKonVol2011} pick elements of $dom(M)$, while the functions mentioned in the definition of $Most^{M,X}$ pick singleton subsets of $dom(M)$.)
 
\cite{DurEbbKonVol2011} prove that the extension of $FO$ with this Most quantifier is not local, in the sense of Theorem \ref{TEOLOCDF}\footnote{The argument for nonlocality is only found in the online archived version of \cite{DurEbbKonVol2011},   arXiv:1109.4750v6. Notice also that in the journal version of the paper, \cite{DurEbbKonVol2015}, the authors give a different definition of Most, which gives rise to local logics.} This shows that, when adding team quantifiers (as defined in this paper) to a local logic, one does not necessarily obtain a local logic.
\end{example}

 We show an example of a team-theoretically nonmonotone quantifier which seems to be captured correctly by our definition.

\begin{example}
We consider the quantifier ``there are finitely many functions which pick countable sets...'', defined as 
\[
\hat\exists_{<\omega,=\omega}^M := \{\textgoth{F}\in \wp(\wp(M)^X) \ | 
card(\textgoth{F}\cap \wp^\omega(M)^X) <\omega \}
\]
where $\wp^\omega(M)$ denotes the set of countable subsets of $dom(M)$.
Consider the usual structure $(\mathbb{N},+,\times)$ of natural numbers, and let $\textgoth P:= \{S\subseteq \mathbb{N} | \forall n\in S(n \text{ is prime})\}$ be the set of all sets of prime numbers. $\textgoth{P}$ is downward closed, in the sense that,  if $P\in\textgoth{P}$ and $P'\subseteq P$, then $P'\in\textgoth{P}$ (this holds because, since $P'\subseteq P\in\textgoth{P}$, $P'$ is also a set of prime numbers). It is known, by Theorem 4.9 of \cite{KonVaa2009}, that each downward closed set of subsets of the domain of a fixed structure is definable by some $IF$ formula of one free variable. Therefore, there is an $IF$ formula $\psi$ (of one free variable $v$) which, on the structure $(\mathbb{N},+,\times)$ for Peano arithmetic, is satisfied exactly by all sets of prime numbers (it defines $\textgoth{P}$). (More rigorously: it is satisfied by all teams $X$ of domain $\{v\}$ such that, for all $s\in X$, $s(v)$ is a prime number). Then, $\psi$ is satisfied by infinitely many (countable) sets of prime numbers: we should expect that $\mathbb{N}\not\models\hat\exists_{<\omega,=\omega}x\psi$. Let us show that our semantic clause does indeed give the expected answer. Observe that here $[\psi]^v_{M,\{\emptyset\}} = \{F:\{\emptyset\}\rightarrow \wp(M) \ | \ M, \{\emptyset\}[F/v] \models \psi\} = \{F:\{\emptyset\}\rightarrow \wp(M) \ | \ \forall a\in F(\emptyset)( a \text{ is prime})\} = \{F:\{\emptyset\}\rightarrow \wp(M) \ | \ F(\emptyset)\in \textgoth{P}\}$. Therefore,  $card(\{F\in [\psi]^v_{M,\{\emptyset\}} \ | \ card(F(\emptyset)) = \omega \}) = card([\psi]^v_{M,\{\emptyset\}}) = card(\textgoth{P}) = 2^{\aleph_0}$, while all $\textgoth{F}$ in $\hat\exists_{<\omega,=\omega}^M$ are such that $\textgoth{F}\cap \wp^\omega(M)^X$ is finite. So, $[\psi]^v_{M,\{\emptyset\}}\notin \hat\exists_{<\omega,=\omega}^M$, which means $\mathbb{N}\not\models\hat\exists_{<\omega,=\omega}v\psi$, as expected. \\
\end{example}

It should be obvious at this point that team quantifiers are an extremely rich family of objects. Perhaps it might be welcome to find some restrictions in order to tame this wild multiplicity. First of all, in order to avoid a number of exceptions, \emph{for the rest of the paper we assume that, if $\textgoth F \in Q^{M,X}$, then $\textgoth F$ does not contain the empty function}.

Secondly, one might want to consider only team quantifiers that are, in some sense, logical. \emph{Invariance by permutations} and the stricter \emph{invariance by isomorphisms} are considered, by many authors, to be necessary conditions for logicality of Mostowski quantifiers. \cite{PetWes2006}, sec. 9.1.1,  describe a recipe for lifting these conditions to objects of higher-order. The key point is that any bijection $g$ between domains $M,M'$ can be lifted to a bijection of higher-order objects. In our case, we first need to lift it to a bijection $g'$ between teams of the same variable domain. Given a team $X$ of domain $V$, we define $g'(X)$ as the team $\{ f \circ s \ | \ s \in X \}$ (where $\circ$ is composition of functions). It is then straightforward to use the recipe of \cite{PetWes2006} to lift $g$ to a bijection $g''$ between sets of supplementing functions (i.e. elements of $\wp(\wp(M)^X)$, for various $X$). Then, invariance by permutations states that, for all $M$, $X$  as above, all permutations $g: M\rightarrow M$, and all $\textgoth F\in \wp(\wp(M)^X)$,
\[
\textgoth F \in Q^{M,X} \iff g''(\textgoth F) \in Q^{M,X}.
\]
Invariance by isomorphisms will be the stricter requirement that, for all bijections $g:M\rightarrow M'$, and all   $X,\textgoth F$ as above,
\[
\textgoth F \in Q^{M,X} \iff g''(\textgoth F) \in Q^{g(M),g'(X)}.
\]
Famously \citep[sec. 3.3.2]{PetWes2006},  for type $(1)$ Mostowski quantifiers, isomorphism invariance is equivalent to a condition on cardinalities of the sets that belong to the (local) quantifiers and their complements. The analogous condition for team quantifiers reads as follows: for all structures $M$ and teams $X$ on $M$,
\begin{quote}
 If $\textgoth{F}\in \hat Q^{M,X}$, $\textgoth{F}'\in \wp(\wp(M)^X)$, $card(\wp(M)^X\setminus \textgoth F) = card(\wp(M)^X\setminus \textgoth F')$ and $card(\textgoth{F}) = card(\textgoth{F}')$, then $\textgoth{F}'\in \hat Q^{M,X}$;
\end{quote}
\noindent that is, if two classes $\textgoth{F}, \textgoth{F}'$ of supplementing functions contain the same number of functions, and also their complements (in the set of all appropriate supplementing functions) contain the same number of functions, then either both $\textgoth{F}, \textgoth{F}'$ belong to $\hat Q^{M,X}$ or neither of them does. Differently from the case of Mostowski quantifiers, this clause is a much stricter restriction than invariance by isomorphisms: it selects quantifiers which only discriminate sets of supplementing functions according to a quantitative component (number of functions) and not to a qualitative component (kind of functions). The quantifier $\hat\exists_{=\hat 3}^{M,X}$ based on \cite{DurEbbKonVol2011} is purely quantitative and satisfies this clause\footnote{Here we are using our nontriviality assumption, that $\textgoth F\in\hat\exists_{=\hat 3}^{M,X}$ entails $\emptyset\notin\textgoth F$.}, but the team quantifier $\hat\exists_{=3}^{M,X}$, based on Engstr\"om's notion of quantifier, does not. More generally, we will see in the next section that each Engstr\"om quantifier $Q$ can be identified with a team quantifier $\hat E(Q)$; save for trivial cases, no quantifier of the form $\hat E(Q)$ satisfies our condition on cardinalities, although many quantifiers of this form are isomorphism invariant. These quantifiers all express the same quantitative component (``There exists a function'') but differ for quality. Also the quantifier $\hat\exists_{<\omega,=\omega}$ is not accepted. Here ``There are finitely many functions'' is the quantity, while ``picking countable sets'' is the quality.

 We might try to formulate a condition on cardinalities which is not so restrictive, and vindicates the role of ``qualities''.
 The idea is that the conditions on cardinality should be imposed only on families of functions that fall within the ``quality'' of the quantifier. We thus need a precise definition of what the ``quality'' of a quantifier is. 
We suggest the following: for each structure $M$ and team $X$ on $M$, let $q(\hat Q,M,X):= \bigcup \hat Q^{M,X} =\{F\in \wp(M)^X | \exists \textgoth{F}\in \hat Q^{M,X} (F\in \textgoth{F})\}$ (the set of all functions that occur in some family of $\hat Q^{M,X}$). Then we decide to accept as (local) team quantifiers only the families $\hat Q^{M,X}\in \wp(\wp(\wp(M)^X))$ such that 
\begin{quote}
 If $\textgoth{F}\in \hat Q^{M,X}$, $\textgoth{F}'\subseteq q(\hat Q,M,X)$, $card(\wp(M)^X\setminus \textgoth F) = card(\wp(M)^X\setminus \textgoth F')$ and $card(\textgoth{F}) = card(\textgoth{F}')$, then $\textgoth{F}'\in \hat Q^{M,X}$.
\end{quote}
This condition is sufficient to exclude some bizarre quantifiers as ``There are $k$ functions which pick sets of cardinality $k$, for some $k$'' (in which quantity and quality are interdependent), while, for example, tolerating the quantifier $\hat\exists_{<\omega,=\omega}$ that we discussed above. If the quantifier $\hat Q$ is not invariant under isomorphism, it seems  reasonable to make our clause stricter by redefining the notion of ``quality'': we do not want simply to take it to be $\bigcup\hat Q^{M,X}$, but also require it to be closed under isomorphisms (lifted to the level of functions $F\in\wp(M)^X$).

\section{Interpreting Engstr\"om's quantifiers} \label{REA}

We want now to interpret and locate more clearly the generalized quantifiers of Engstr\"om (as reviewed in section \ref{SECENG}) within the framework for generalized team quantifiers that we introduced in section \ref{SECSECOND} In particular, we want to understand to what extent Engstr\"om's proviso that his definition be applied only to monotone quantifiers is justified. 

Concerning this last point, Engstr\"om only offers an example in defense of this restriction:

\begin{quote}
(...) This applies even for non-monotone quantifiers but for those quantifiers $Q$ the
truth condition above does not make a whole lot of sense as the following example
shows. Let $M = \mathbb{N}$ and $Q = \{A\}$ where $A$ is the set of even numbers. According
to the truth condition above $M, \{\emptyset\}\models Qx(x = x)$ since there is a team $X = A(x)$
such that $M,X \models x = x$. \citep[sec. 2; we slightly changed the notation]{Eng2012}
\end{quote}

Engstr\"om does not provide any explanation why the quantifier in this example seems to him to be treated incorrectly. Our guess is that, in analogy with the semantic clause of Mostowski, he expects a sentence $Qx\psi$ to be satisfied if and only if the set of elements which satisfy $\psi$ is \emph{exactly} $A$. This is not the case in the example, for $x=x$ is satisfied by all elements of $\mathbb{N}$. So, it seems that the worries of Engstr\"om are driven by the desire to obtain conservative extensions over logics which accomodate Mostowski quantifiers. Such a conservative extension result has been proved for \emph{sentences} of Mostowski extensions of first-order logic \citep[Proposition 2.4.4]{Eng2012} and then generalized to an embedding of Mostowski extensions of existential second-order logic\footnote{In line with our earlier conventions, we might denote such extensions as $ESO^Q$. The notation used in \cite{EngKon2013} was $ESO(Q)$.} into corresponding extensions of Dependence logic \citep{EngKon2013}. The following example shows in what way the interpretation of Engstr\"om quantifiers is deviant, over open formulas, with respect to Mostowski quantifiers, and compares it to the interpretation via team quantifiers (as defined in the previous section).

\begin{example}
Let us consider the sentence $\theta:\forall x\exists_{=3}yP(x,y)$, which can be seen both as a (generalized) first-order sentence and as a (generalized) $IF$ sentence. According to first-order semantics and Mostowski's definition,
\[
M\models \forall x\exists_{=3}yP(x,y) \Leftrightarrow \text{for all } s \text{ s.t. } dom(s) =\{x\}, M,s\models \exists_{=3}yP(x,y) 
\]
\[
\Leftrightarrow \text{for all } s \text{ s.t. } dom(s) =\{x\}, \{a\in M \ | \ M,s(a/y)\models P(x,y)\} \in \exists^M_{=3}.
\] 
This is not
 the usual notion of supplementing (and Engstr\"om's definition for the non-monotone cases is probably intended to bridge this gap). Instead, the team semantics of the universal quantifier, plus Engstr\"om's semantics for generalized quantifiers, yields: 
\[
M\models \forall x\exists_{=3}yP(x,y) \Leftrightarrow M,\{\emptyset\}[M/x]\models \exists_{=3}yP(x,y) 
\]
\[
\Leftrightarrow \text{there is a function } F: \{\emptyset\}[M/x]\rightarrow \exists_{=3}^M \text{ s.t. } M,X[MF/xy]\models P(x,y).
\] 
This is a different meaning; it states that for all assignments $s$ s.t. $dom(s) =\{x\}$, and for a specific function $F$, it holds that $\{F(s) \ | \ M,s(F(s)/y)\models P(x,y)\}\in \exists^M_{=3}$. This is a weaker condition; it does not require that the set $[\theta]_{M,s}$ of all $a\in M$ such that $s(a/y)$ satisfies $P$ to be an element of $\exists^M_{=3}$; but only that, for each $s$, we can pick a \emph{subset} of $[\theta]_{M,s}$ that falls in $\exists^M_{=3}$ (and whose assignments all satisfy $P(x,y)$). So, under team semantics, the very same sentence is true in more models than in the first-order case.   
\end{example}

Remember from example \ref{EXTHREE} that formulas of the form $\exists_{=3}\psi$ are in general equivalent to team quantifier expressions of the form $\hat\exists_{=3}\psi$.  
 The fact that the quantifier $\hat\exists_{=3}$  can \emph{also} be expressed in terms of the first-order local quantifier $\exists_{=3}^M$ (similarly to what happened for  $\hat\exists,\hat\forall$) is a peculiar accident, which plausibly does not apply to most team quantifiers. But we will see that each Engstr\"om quantifier (irrespectively of monotonicity properties) has a corresponding team quantifier. Before that, let us add some further remarks on the example of Engstr\"om that was quoted at the beginning of the section. One fishy aspect of this example is that it discusses an $FO(Q)$ extension; in particular, one key element in Engstr\"om's argument seems to be the fact that $x=x$ is a flat formula, so that it makes sense to say that it is satisfied by this or that number. As we have seen in the previous section, in the more general context of team semantics, if a formula is not flat it makes no sense to say it is satisfied by an element of the domain (resp. a tuple of elements). 
 It is not a surprise, then, that the above-mentioned conservativity results only hold at the level of sentences.\footnote{It is instead surprising that such a conservativity result does hold at the level of open formulas, if one considers quantifier extensions of the so-called Independence logic $I$ \citep[Theorem 4.1]{EngKon2013}. Yet the correspondence is still imperfect in some ways; most notably, the $FO^Q$ fragment of $ESO^Q$ does not translate into the $FO(Q)$ fragment of $I(Q)$.} Notice that the example is perfectly meaningful if we read the quantifier $Q$ in a second-order fashion (which is justified, since the quantifier occurs sentence-initially -- see the remarks in the previous section). Under this reading, the sentence $Qx(x=x)$ in the example states that ``there is a set, containing exactly the even numbers, which satisfies $x=x$''; it is then reasonable that $Qx(x=x)$ be true in $\mathbb{N}$ since the set of even numbers (or, more precisely, the corresponding team $X=\{\{(x,2n)\} \ | \ n \in \mathbb{N}\}$) does satisfy $x = x$. Within this reading of the quantifiers, using the rough analysis of the previous sections, ``there exists'' is the \emph{quantity}, while ``containing exactly the even numbers'' is the \emph{quality} of the quantifier. Underlying these observations is the claim that all the Engstr\"om quantifiers can be thought, instead, as \emph{existential} second-order quantifiers (their ``quantity'' is just the statement of the existence of a \emph{nonempty} set). This claim can be stated in a more general form and proved:


\begin{thm} \label{TEOEXLIFT}
Let $Q$ be a Mostowski quantifier, and consider the second-order quantifier $\hat E(Q)$ (for brevity: $\hat Q$) given by the condition, for all domains $M$ and teams $X$:
\[
\hat Q^{M,X} := \{\textgoth{F}\in \wp(\wp(M)^X) \ | \ \exists F\in \textgoth{F}\forall s\in X (F(s)\in Q^M)\}
\]
Then, for every structure $M$, team $X$ and $\psi$ $IF(\mathcal Q \cup \{Q\})$ formula,
\[
M,X \models \psi \iff M,X \models \psi^*
\]
where $\psi^*$ is the $IF(\mathcal Q)[\hat Q]$ formula obtained from $\psi$ by replacing all occurrences of $Q$ with $\hat Q$. The same statement holds for extensions of $DF$ logic, or for formulas with occurrences of both slashed and backslashed quantifiers.
\end{thm}

\begin{proof}
By induction on the syntax of $\psi$; the only nontrivial case  is $\psi = (Qv/V)\chi$ (the case $\psi = (Qv\backslash V)\chi$ is analogous).

Suppose $M,X\models \psi$. Then there is a $V$-uniform function $F:X\rightarrow Q^M$ s. t. $M,X[F/v]\models \chi$. By the inductive assumption, $M,X[F/v]\models \chi^*$ So, $F\in [\chi^*]^{v,V}_{M,X}$. So $[\chi^*]^{v,V}_{M,X}\in \hat Q^{M, X}$. Therefore $M,X\models (\hat Qv/V)\chi^*$, i.e., $M,X\models\psi^*$.

Vice versa, assume $M,X\models (\hat Qv/V)\chi^*$. This means that $[\chi^*]^{v,V}_{M,X}\in \hat Q^{M, X}$. But then there is a function $F\in [\chi^*]^{v,V}_{M,X}$ such that  $F(s)\in Q^M$ for each $s\in X$; this latter condition states that (1): $F$ is a function $X\rightarrow Q^M$. $F$ being an element of $[\chi^*]^{v,V}_{M,X}$ means that (2): $F$ is $V$-uniform and that $M,X[F/v]\models\chi^*$. By inductive hypothesis, we have (3): $M,X[F/v]\models\chi$. The statements (1), (2) and (3) are the semantical conditions for asserting that $M,X\models (Qv/V)\chi$.
\end{proof}

So, each Engstr\"om quantifier $Q$ is equivalent to some team quantifier $\hat E(Q)$ which is existential and monotone in the team-theoretical sense; the informed reader can probably see the analogy between this operator $E$ and the so-called existential lift which can be used to produce a determiner with a collective reading starting from a determiner with distributive reading \citep{Sch1981,Doe1993}. 

In case 1) the Engstr\"om quantifier $Qv$ is monotone (in the first-order sense), and 2) the formula $\psi$ to which $Q$ is applied is flat, the meaning of the formula $Qv\psi$ (or its equivalent $E(Q)v\psi$)  coincides with that of its first-order equivalent (where $Q$ is interpreted as a Mostowski quantifier), in the sense that, for all $M$ and $X$:
\[
M,X\models Qv\psi \iff \text{for all } s\in X, \ M,s\models Qv\psi
\]
 (on the left we have team semantics with the clause of Engstr\"om, on the right Tarskian semantics with the clause of Mostowski). The following example shows what goes wrong if either 1) or 2) is not satisfied.


\begin{example}
1) Let us consider the ``most'' quantifier of \cite{Sev2014}, which is monotone (in the first-order sense). For simplicity, consider just an occurrence of this quantifier at the beginning of a sentence. Thinking in the framework of team quantifiers, the proper interpretation of $M,\{\emptyset\}\models \operatorname{most} y\psi$ in $IF$ logic is  that ``there is a set (of $y$s), containing most elements of $M$, such that $\psi$ holds \emph{of} this set'' (not \emph{on} this set, which would mean that $\psi$ holds of each single element in the set; we are treating $\psi$ as a \emph{global}, or \emph{collective}, property of the set). If $\psi$ is a flat $IF$ formula (of one free variable), this assertion is equivalent to ``$\psi$ holds of most $y$s''. The reason    
is simply that the semantical clause of Engstr\"om tells us that there is a set that contains most elements of $dom(M)$ and that satisfies $\psi$; flatness tells us that each element of that set satisfies $\psi$; and by monotonicity then the set of \emph{all} individuals which satisfy $\psi$ is a set containing most of the elements of $dom(M)$ (the converse is straightforward). Instead, 
 for formulas that are not flat, it simply makes no sense to say that ``$\psi$ holds for most elements'', and we see no clear way to compare the behavior of this ``most'' quantifier in this context with the classical behaviour of its Mostowskian counterpart.

 2) Consider the quantifier $\exists_{\leq 2}$, such that $\exists_{\leq 2}^M$ is the set of at-most-two-element subsets of $M$. It is not upwards monotone; we can see now that, if a quantifier is not upwards monotone, the interpretation of sentences can easily diverge from the intended first-order meaning, even when the quantifier is applied to flat formulas. Indeed,  $\exists_{\leq 2}$  is clearly  not the quantifier ``there are at most two''; this is because any  sentence $\exists_{\leq 2}y\psi$, with $\psi$ first-order, is trivially true in team semantics (since $\psi$ is satisfied by $\emptyset$, see Theorem \ref{EMPTYSETR}). Thus, also its equivalent team quantifier $\hat E(\exists_{\leq 2})$ is trivial over $FO$, $IF$ or $DF$ logic.\footnote{The literature is not devoid of team-theoretical logics which do not satisfy the empty team property; for example, such logics are obtained by extending $FO$, $IF$ or $DF$ with an operator for contradictory negation. The interpretation of $\exists_{\leq 2}$ and $\hat E(\exists_{\leq 2})$ might not be trivial over such logics.}

Similarly, the non-monotone quantifier $\exists_{=2}$  (such that $\exists_{=2}^M$ is the set of two-element subsets of $M$) does not have the classical interpretation.
Let $\psi$ be a first-order formula of one free variable $x$. Clearly, according to team semantics the sentence $\exists_{=2}x\psi$ (resp. $\hat E(\exists_{=2})$) is satisfied in any domain where $\psi$ applies to \emph{at least} two elements, and this is not the intended first-order meaning. When the quantifier is sentence-initial, it does respect the second-order meaning (``there is a set of $2$-elements'').
\end{example}

We open a parenthesis in order to understand better the distinction we made above between a property holding \emph{of} a set, or instead \emph{on} a set (collective vs. distributive reading). We may consider an analogy with Aczel's third-order characterization of (first-order) sentence-initial generalized quantifiers \citep{Acz1975}. Let us focus again on the case of a quantifier in sentence-initial position. If $\textgoth P$ is a second-order unary predicate (applying to one variable for unary first-order predicates), we can define the first-order generalized quantifier $Q$ by a third-order clause:
\[
M\models Qz\psi \Leftrightarrow (M,Q^M)\models\exists S(\textgoth P(S) \land \forall z(S(z)\rightarrow \psi(z))).
\]
where $(M,Q^M)$ is a shorthand for the \emph{second}-order structure which assigns the local quantifier $Q^M$ as interpretation for the second-order predicate $\textgoth P$.
This clause combines $Q$ with the \emph{on} reading of $\psi$: $\psi$ holds of every element of the set $X$. How can we express instead the \emph{of} reading, and thus the semantics of Engstr\"om quantifiers? We just need to raise Aczel's definition to a fourth-order formula. Let $\hat{P}$ be a third order predicate applying to a second-order unary predicate of the same type as the $\textgoth P$ above; 
 then, one can define a quantifier over sets as:
\[
M\models Qz\psi \Leftrightarrow (M,|\hat Q|^M)\models\exists \textgoth{P}(\hat{P} (\textgoth{P}) \land \forall S(\textgoth{P}(S)\rightarrow \tau_\psi(S))).
\]
where $(M, |\hat Q|^M)$ is the \emph{third} order structure assigning, as interpretation for the third-order predicate symbol $\hat{P}$, the sentence-initial version of the local team quantifier $\hat Q = \hat E(Q)$ (as defined in the previous section); and $\tau_\psi(S)$ is the translation, due to \cite{Hod97b}, of the $IF$ formula $\psi$ into an existential second-order sentence with the extra predicate symbol $S$. (We cannot simply write $\psi(S)$, because an $IF$ formula applies to individual variables, not variables for predicates).
This equivalence captures the \emph{of} meaning ($\psi$ expresses something about $X$, not about its elements).

Let us return to the analysis of the Engstr\"om quantifier. We have a further argument against the restriction of such quantifiers to the monotone case; it is a proof-theoretical observation. The point is that a good number of the known inference rules of $IF$ logic which involve existential quantifiers hold in some similar form for Engstr\"om quantifiers; we analyze a number of these, and their consequences, in the Appendix. The monotonicity restriction does not affect any of the rules we examined; in other words, we could not find any proof-theoretical discriminant between the monotone and the nonmonotone case. Similarly, on the semantical side, we saw (Theorem \ref{TEOLOCDF} and following observations) that whether a generalized quantifier extension of a local logic has the locality property or not does not dependend on monotonicity (the relevant restriction is union closure, or, in the case of downward closed logics, no restriction at all).


\section{Capturing the non-monotone quantifiers}   \label{SECNONMON}

In the previous section, we saw that the monotone Engstr\"om quantifiers happen to have the usual, first-order meaning whenever they are applied to first-order formulas; and also, that this fails to be the case if the quantifier is non-monotone. One may wonder whether there is some different notion of non-monotone quantifier in team semantics which, when applied to a first-order formula, gives the same results as would obtain by Tarskian semantics. Somewhat surprisingly, the answer is yes.
Engstr\"om gave a second semantical clause, meant to cover both the monotone and non-monotone quantifiers. We briefly recall it; we will write $\models^b$ (for ``bounded'') for the satisfaction relation of extensions $FO^b(\mathcal Q)$ , $IF^b(\mathcal Q)$ or $DF^b(\mathcal Q)$ in which the generalized quantifiers are interpreted in this way. For functions $F,F':X\rightarrow \wp(M)$, write $F\leq F'$ if, for all $s\in X$, $F(s) \subseteq F(s')$. Then, Engstr\"om's semantical clause reads: $M,X\models^b Qx\psi$ iff there is a function $F:X\rightarrow \wp(M)$ such that
\begin{quote} \label{SECONDDEF}
(1) $M,X[F/x]\models \psi$, and\\
(2) for each $F'\geq F$, if $M,X[F'/x]\models \psi$ then for all $s\in X: F'(s)\in Q^M$.\footnote{Actually in \cite{Eng2012} the second clause ends with ``$F(s)\in Q$'', but reading the paragraphs that follow, it becomes clear that the reference to $F$ instead of $F'$ is a typo.} 
\end{quote}
This clause is successful in its purpose: 1) on monotone quantifiers, it gives the same results as the previous definition \citep[Prop.2.10]{Eng2012}, and 2) a $FO^b(Q)$ formula $\psi$ is satisfied by a team $X$ if and only if each assignment $s\in X$ satisfies $\psi$ in the Mostowski sense \citep[Prop.2.11]{Eng2012}.

However, the reader would probably not be surprised, at this point, to see that even this more refined definition may fail to capture the intuitive first-order reading of the quantifier when the quantifier itself is applied to an $IF$ formula. 

\begin{example}
 Consider the quantifier $\exists_{=\omega}$ such that $\exists_{=\omega}^M = \{S\subseteq dom(M) \ | \ card(S) = \omega\}$ (in the first-order reading, ``there are countably (and not finitely) many''). Now, since $IF$ logic is capable of talking of bijections, we should expect to be capable of expressing, in $IF(\exists_{=\omega})$, the concept $C$ that ``there are countably many $x$ that can appear in the domain of some injective function of codomain $P$''. The mere existence of a bijection of the domain $M$ with $P^M$ can be expressed by the $IF$ sentence
\[
\forall x\forall y(\exists u/\{y\})(\exists v/\{x,u\})(x=y \leftrightarrow u=v\land P(u)).\footnote{In order to fit with our syntax, the subformula $x=y \leftrightarrow u=v$ should be replaced by some classical equivalent expressed in terms of $\land,\lor$ and negation of atoms, such as $(x=y \land u=v)\lor(x\neq y \land u\neq v)$. \emph{Any} such equivalent will do, thanks to the fact that this is a first-order formula.}
\]

This can be most easily seen using the Skolemization procedure for $IF$ sentences \citep[sec. 4.3]{ManSanSev2011}. 
Treating $u$ as a function of $x$ ($f(x)$) and $v$ as $g(y)$, the $IF$ sentence is equivalent to the existential second-order sentence
\[
\exists f\exists g \forall x\forall y (x=y \leftrightarrow f(x) = g(y)\land P(f(x))).
\]
Indeed, $x=y \rightarrow f(x) = g(y)$ states that $f$ and $g$ denote the same function; $f(x) = g(y) \rightarrow x=y$ forces this function to be injective; by $P(f(x)))$, the image of the function is contained in $P^M$.

So, it seems reasonable that the existence of countably many $x$ that can be in the domain of an injection in $P$ be expressed by 
\[
\varphi: \exists_{=\omega} x\forall y(\exists u/\{y\})(\exists v/\{x,u\})(x=y \leftrightarrow u=v\land P(u)).
\]

  If $(M,P^M)$ is a structure of uncountable domain, with $P^M\subset dom(M)$ a countable set, the statement $C$ should be false (because there are uncountably many such $x$: all $x\in dom(M)$). Yet, the formula $\varphi$ is true in $M$. Indeed, let $F:\{\emptyset\}\rightarrow dom(M)$ be any function such that $card(F(\emptyset)) = \omega$; fix a bijection $g: F(\emptyset)\rightarrow P^M$ and an element $a\in P^M$  and let $J:\{\emptyset\}[FM/xy]\rightarrow \wp(M)$ be $J(s) := \{g(s(x))\}$, if $s(x)\in F(\emptyset)$, and $J(s) := a$, otherwise. Similarly, define $K:\{\emptyset\}[FMJ/xyu]\rightarrow \wp(M)$ by $K(s) := \{g(s(y))\}$, if $s(y)\in F(\emptyset)$, and $K(s) := a$, otherwise. It should be clear that $J$ is $\{y\}$-uniform, $K$ is $\{x,u\}$-uniform, and $M,\{\emptyset\}[FMJK/xyuv] \models (x=y \leftrightarrow u=v\land P(u))$; the semantical clauses then yield $M,X[F/x]\models\forall y(\exists u/\{y\})(\exists v/\{x,u\})(x=y \leftrightarrow u=v\land P(u))$.      We still have to verify that the clause 2) of the semantics of $\exists_{=\omega}$ is true. Suppose for the sake of contradiction that there is a function $F':\{\emptyset\}\rightarrow dom(M)$, such that $F'\geq F$\footnote{Notice that the assumption $F'\geq F$ in the argument that follows.}, $card(F'(\emptyset))>\omega$ and $M,X[F'/x]\models\forall y(\exists u/\{y\})(\exists v/\{x,u\})(x=y \leftrightarrow u=v\land P(u))$. Then there are functions $J',K'$ such that $M,\{\emptyset\}[F'MJ'K'/xyuv] \models (x=y \leftrightarrow u=v\land P(u))$. But $M,\{\emptyset\}[F'MJ'K'/xyuv]\models x=y \leftrightarrow u=v$ entails that $card(\bigcup\operatorname{Im} J')\geq card(F'(\emptyset))>\omega$. Instead, $M,\{\emptyset\}[F'MJ'K'/xyuv]\models P(u)$ entails that $\bigcup\operatorname{Im} J'\subseteq P^M$, so that $card(\bigcup\operatorname{Im} J')\leq card(P^M) = \omega$, and we have a contradiction. Therefore, if $F'\geq F$ and $M,\{\emptyset\}[F'/x]\models\forall y(\exists u/\{y\})(\exists v/\{x,u\})(x=y \leftrightarrow u=v\land P(u))$, we have $F'(\emptyset)\in \exists_{=\omega}^M$. We conclude that $M\models^b \varphi$. 

The seemingly paradoxical nature of this result disappears as soon as we read $\exists_{=\omega}$ in its team theoretical interpretation; then we can see that $\varphi$ just states the fact (true precisely in any domain such that $card(P^M) = \omega$) that \emph{there are countable sets} that can be injected in $P^M$, and no larger set can. It is the collective, and not the distributive reading which is at work here.

The concept $C$ is in general equivalent to the statement that both $dom(M)$ and $P^M$ are countably infinite.
It must be remarked that we cannot expect to be able to define the concept $C$ in $IF$ logic, because it is well-known that $IF$ logic (which has the L\"owenheim-Skolem property, and can express bijections) cannot define the countability of a set. We might wonder whether it is possible at all to express $C$ in quantifier extensions of $IF$ logic. We see no straightforward way to do this by means of Engstr\"om quantifiers. But there is a team quantifier that does the job. Consider the local team quantifier
\[
\hat Q_{\omega set}^{M,X}:= \{\textgoth F \in \wp(\FU{X}{\wp(M)}) \ | \ \exists F\in \textgoth{F}\exists s\in X(card(F(s))=\omega) 
\]
\[
\text{ and } \forall F\in \textgoth{F}\forall s\in X(card(F(s))\leq\omega)\} 
\] 
Then $C$ is captured by $Q_{\omega set}v (v=v) \land Q_{\omega set}v P(v)$, as can be easily verified. 
\end{example}

The example above shows that the interpretation of quantifiers is problematic also with the semantics given by $\models^b$. It is however straightforward to see that these kinds of quantifiers, as those given by Engstr\"om's earlier definition (Theorem \ref{TEOEXLIFT}), are identifiable with appropriate team quantifiers, by means of a different lift $\hat B: Q \mapsto \hat B(Q)$ (``bounded lift''). 

\begin{thm} \label{TEOAELIFT}
Let $Q$ be an Engstr\"om quantifier, and consider the second-order quantifier $\hat B(Q)$  given by the condition, for all domains $M$ and teams $X$:
\[
\hat B(Q)^{M,X} := \{\textgoth F \in \wp(\FU{X}{\wp(M)}) \ | \ \exists F\in \textgoth{F}\forall s\in X(F(s) \in Q^M \land \forall F'\geq F(F'\in \textgoth{F} \Rightarrow F'\in Q^M))\} 
\] 

Then, for every structure $M$, team $X$ and $\psi$ $IF(\mathcal Q \cup \{Q\})$ formula,
\[
M,X \models^b \psi \iff M,X \models \psi^*
\]
where $\psi^*$ is the $IF(\mathcal Q)[\hat B(Q)]$ formula obtained from $\psi$ by replacing all occurrences of $Q$ with $\hat B(Q)$. The same statement holds for extensions of $DF$ logic, or for formulas with occurrences of both slashed and backslashed quantifiers.
\end{thm}

\begin{proof}
We give the argument for $IF$. We reason by induction on the syntax of $\psi$; the significant case is $\psi = (Qv/V) \chi$. Assume first that $M,X\models^b\psi$. Then there is a $V$-uniform function $F: X \rightarrow Q^M$ such that 1) $M,X \models^b \chi$, and 2) for all $F'\geq F$ s.t. $M,X \models^b \chi$, it holds that $\forall s \in X(F'(s)\in Q^M)$. By 1) and the inductive hypothesis, $F\in [\chi^*]^{v,V}_{M,X}$. Suppose then that $F'\geq F$ and $F'\in [\chi^*]^{v,V}_{M,X}$. By the inductive hypothesis we have $F'\in [\chi]^{v,V}_{M,X}$; this, together with 2) and $F'\geq F$ gives that, for all $s\in X$, $F'(s)\in Q^M$. So $[\chi^*]^{v,V}_{M,X}\in \hat B(Q)^{M,X}$, that is, $M,X\models (\hat B(Q)v/V)\psi$.

The argument in the opposite direction is similar.
\end{proof}

So, also the quantifiers interpreted according to Engstr\"om's second definition can be identified with appropriate team quantifiers. We wish to point out that one of the reasons for the deviation of these quantifiers from first-order meanings seems to be the fact that the condition $F'\in \textgoth{F} \Rightarrow F'\in Q^M$ is applied only to functions $F'\geq F$; it would seem reasonable to apply it to all functions. We may indeed try to build a counterexample in which a formula of the form $\exists_{=\omega} x\psi$ is true because there is a function $F$ satisfying 1) and 2) (for the quantifier $\exists_{=\omega}$); but at the same time there is a function $G$ which satisfies 1) but not 2). This means that there is a $G'\geq G$ such that $card(G'(\emptyset))>\omega$ and $M,\{\emptyset\}[G'/x]\models \psi$. Let $H$ be the function such that $H(\emptyset) = F(\emptyset) \cup G'(\emptyset)$. Now, in case $\psi$ is flat, it immediately follows that  $M,\{\emptyset\}[H/x]\models \psi$; so $H\geq F$ and $M,X[H/x]$, but $card(H(\emptyset))>\omega$: condition 2) is contradicted. However, if $\psi$ is not flat, the argument does not carry over; we build here a concrete counterexample.

\begin{example}
Let $M=(dom(M),<)$ be a structure which interprets $<$ as a partial order such that 1) it has a minimum element $a$, 2) it has two maximal chains (linear suborders) ending, respectively, in maximal elements $b_1$ and $b_2$, and 3) the first maximal chain is infinitely countable, while the second maximal chain is uncountable. Call $C_1$ and $C_2$ the sets of elements in the first and, respectively, in the second chain. To make things easier, we may also assume that $C_1\cap C_2 = \{a\}$ is the only point in common. Now, the property ``$R$ is a chain'' is expressed in $M$ by the first-order sentence (of signature $\{<,R\}$) : $\forall x \forall y(R(x,y)\lor R(y,x)\lor x=y)$. And it is a downward closed property: every subset of a linear order is a linear order. Then, Theorem 4.9 of \cite{KonVaa2009} guarantees the existence of a formula $\psi(x,y)$ which is satisfied by a team $X$ on $M$ iff $X(x,y)$ is a linear order. Consider the formula $\exists_{=\omega}x(\exists y_{=\omega}/\{x\})\psi(x,y)$.  The function $F(\emptyset):= C_1$ is such that: 1a) $M, \{\emptyset\}[FF/xy]\models \psi(x,y)$; 2a) for any function $F'>F$, $M, \{\emptyset\}[FF'/xy]\not\models \psi(x,y)$ (let $c\in C_1\setminus \{a^M\}$ and $d\in F'(\emptyset)\setminus F(\emptyset)$; the singleton subteam $\{(x,c),(y,d)\}$ does not satisfy $\psi(x,y)$, and so by downward closure neither does $\{\emptyset\}[FF'/xy]$);  1b) $M, \{\emptyset\}[F/x]\models \exists_{=\omega}x\psi(x,y)$ (by 1a and 2a); 2b) if $F'>F$, $M, \{\emptyset\}[F'/x]\not\models \exists_{=\omega}x\psi(x,y)$ (by a similar argument as for 2a) ). Thus,  $M\models\exists_{=\omega}x(\exists y_{=\omega}/\{x\})\psi(x,y)$. Notice however that also the function $G(\emptyset):=C_2$ (which picks uncountably many elements) is such that $M, \{\emptyset\}[G/x]\models (\exists y_{=\omega}/\{x\})\psi(x,y)$ (by similar arguments).
\end{example}

The example reveals that the quantifier $\exists_{=\omega}$, interpreted according to $\models^b$, respects the team-theoretical intuition that ``there is a countable set'' which satisfies the subformula $(\exists y_{=\omega}/\{x\})\psi(x,y)$; but not the first-order intuition that ''there are countably many elements'' (the same formula can be satisfied using uncountably many elements). If we eliminate the restriction $F'\geq F$ in condition 2), and redefine the bounded lift as
\[
\hat B'(Q)^{M,X} := \{\textgoth F \in \wp(\FU{X}{\wp(M)}) \ | \
\textgoth{F}\neq \emptyset  \ \land \, \forall F'(F'\in \textgoth{F} \Rightarrow F'\in Q^M)\}. 
\] 
we obtain a notion of quantifier that comes closer to first-order intuitions, but is more questionable from a team-theoretical perspective. The quantifier $\hat\exists_{=3,nm} = \hat B'(\exists_{=3})$, considered in example \ref{EXTHREE}, has this form.

Finally, we contrast the behaviour of the semantical clause considered in this section with Engstr\"om's earlier one under the aspect of permutation of quantifiers. Engstr\"om showed that, under $\models^b$, generalized quantifiers may fail to abide to some reasonable swapping rule, such as the equivalence of $Qu(Q'v/\{u\})$ with $Q'v(Qu/\{v\})$. Instead, in the Appendix we show that, save for some trivial quantifiers, the permutation rules work correctly for quantifiers evaluated according to Engstr\"om's earlier clause, even in the nonmonotone case.




\section{Conclusions}

One of the main points which emerged from our analysis of Engstr\"om's definitions of generalized quantifiers (his first definition, restricted to monotone quantifiers; and the second clause for non-monotone ones) is that these must be taken as a clever idea to embed the \emph{first-order} generalized quantifiers of Mostowski into team semantics; but the correctness of this embedding is limited to quantifiers applied to flat formulas. Since the typical logics based on team semantics are not flat, the appropriateness of these quantifiers for logics of imperfect information is difficult to assess.

Considerations over the higher-order nature of team semantics lead us to conjecture that an appropriate notion of generalized quantifier for team semantics should use semantical objects which are of higher order than those involved in the semantics of the Mostowski quantifiers. We proposed such a definition of ``team quantifier'', which treats each quantifier as a set of sets of functions; this definition includes, as special cases, the two definitions of Engstr\"om and the Most quantifier of \cite{DurEbbKonVol2011}. The identification of Engstr\"om quantifiers with team quantifier is performed via two operators that we called, respectively, the existential lift $\hat E$ and the bounded lift $\hat B$. 

Importantly, the existential lift is correct also in case it is applied to a quantifier which is (from a first-order perspective) nonmonotone. Engstr\"om instead stated that his first definition is not applicable to nonmonotone quantifiers. Many of the arguments in this paper point to the fact that this restriction is artificial, and induced by the mistake of treating a quantifier $\hat E(Q)$, which is higher-order in content, as if it were a first-order quantifier. The quantifiers $\hat B(Q)$ have been similarly misinterpreted.

We sketched some basic observations on team quantifiers, such as some discussion of their logicality. One obvious disadvantage of our definition is that it operates quite high in the type hierarchy (e.g. a team quantifier is a fourth-order object), but still we think it is somewhat natural. This claim is supported, first, by the arguments that lead us to this definition by analogy with the semantics of Mostowski quantifiers; and secondly, by our observation that, for sentence-initial quantifiers, our definition collapses to the more well-studied notion of second-order quantifier, which treats the meaning of a $1$-variable formula as a set of sets -- exactly as team semantics does, through the identification of a $1$-variable team with a subset of the domain of discourse.

\begin{center}
* * *
\end{center}

\appendix
\section{Appendix: equivalence rules for $IF^*(\mathcal R)$}



In section \ref{REA} we argued that there is no reason to restrict Engstr\"om's first definition of generalized quantifier to the monotone case. We further defend this idea by showing that some good logical properties (in particular, prenex transformations and the primality test) apply to extensions $IF(\mathcal R)$ without regard for the monotonicity, or lack thereof, of the quantifiers in $\mathcal R$. 

It is sometimes easier to study the proof-theoretical aspects of $IF$ logic if one considers a more general syntax  which also allows slashed connectives. We follow the convention of \cite{CaiDecJan2009} in calling this system $IF^*$ logic. The syntax is obtained by replacing the clauses for conjunction and disjunction of $IF$ logic with the clauses $\psi \land_{/W}\chi$ and $\psi \lor_{/W}\chi$ (where $W$ is a finite set of variables, $\psi,\chi$ $IF^*$ formulas). All the other syntactical clauses must also be extended to apply to $\psi,\chi$ $IF^*$ formulas. The definition of the set of free variables of an $IF^*$ formula requires the additional clauses: $FV(\psi \land_{/W}\chi) = FV(\psi \lor_{/W}\chi)= FV(\psi)\cup FV(\chi)\cup W$. Finally, it is necessary to add two clauses to the semantics. We say that a subset $Y$ of a team $X$ is \textbf{$W$-uniform} in $X$ if $s\in Y$, $s'\in X$ and $s\sim_W s'$ imply that $s'\in Y$. Then: 

\begin{itemize}
\item $M,X\models \chi_1\land_{/W}\chi_2$ if $M,X\models\chi_1$ and $M,X\models\chi_2$  
\item $M,X\models \chi_1\lor_{/W}\chi_2$ if there are $Y,Z$ that are $W$-uniform subsets of $X$, and such that $Y\cup Z=X$, $M,Y\models \chi_1$, and $M,Z\models \chi_2$.
\end{itemize}  

The $IF^*$ language is a conservative extension of the $IF$ language, provided one identifies $\land_{/\emptyset}$ with $\land$ and $\lor_{/\emptyset}$ with $\lor$. Generalized quantifier extensions $IF^*(\mathcal R)$ are defined as in the case of $IF$ logic (section \ref{SECENG}). Remember that $\mathcal R$ denotes a set of quantifiers distinct from $\forall$ and $\exists$.

 In the following, we will prove the validity of several equivalence rules of $IF^*(\mathcal R)$; the reader can infer from each of them a corresponding rule for $IF(\mathcal R)$. Our proofs will follow the model of \cite{CaiDecJan2009}, although now we must take care of the fact that functions from a team to $\wp(M)$ are considered; and of the strange things that may happen if some function has the empty set as one of its values. (In the monotone case, the only quantifier which is affected by this exception is the trivial quantifier $T^M = \wp(\wp(M))$).

\subsection*{Further notational conventions}

We list some additional notational conventions that will be used in this appendix. Most of them are borrowed from \cite{CaiDecJan2009}.
\begin{itemize}
\item We omit union symbols in syntactical expressions, e.g. we write $W\cup\{v\}$ as $Wv$. 
\item We write $\psi_{/V}$ for the $IF^*$ formula which  is obtained from $\psi$ by adding the variables of $V$ to each of the slash sets (including the slash sets of disjunctions). If $V=\{v\}$, we simply write $\psi_{/v}$.
\item Similarly,  we write $\psi|_{V}$ (resp. $\psi|_{v}$) for the $IF^*$ formula which  is obtained from $\psi$ by adding the variables of $V$ (resp. the single variable $v$) to each of the \emph{nonempty} slash sets (including the slash sets of disjunctions).
\item We denote by $s[z/x]$ the assignment which is obtained replacing each element of the form $(x,a)$ with an element $(z,a)$.
\item We denote as $X_{[z/x]}$ the team $\{s[z/x] | s\in X\}$.
\item We denote by $\psi[z/x]$ the formula obtained from $\psi$ by replacing each free occurrence of $x$ with $z$. 
\item If $\chi$ is an occurrence of a subformula of $\psi$, we denote as $\psi(\theta/\chi)$ the formula obtained replacing $\chi$ with $\theta$ in $\psi$.
\item A team $X_v$ is called a \textbf{$v$-expansion} of team $X$ if $v\notin dom(X)$, $dom(X_v) = dom(X) \cup \{v\}$ and $(X_v)_{\restriction dom(X)}=X$.
\end{itemize}
We will state inference rules in terms of \textbf{$Z$-equivalence}\footnote{This notion of equivalence must not be confused with the alternative approach of ``relative equivalence'' pursued in \cite{ManSanSev2011}, or with the special case of ``safe equivalence'' of \cite{Dec2005}.}: given a finite set of variables $Z$, two $IF^*$ formulas $\psi$ and $\chi$ are said to be $Z$-equivalent, in symbols $\psi \equiv_{Z}\chi$, if:
\begin{quote}
1) $(FV(\psi)\cup FV(\chi))\cap Z = \emptyset$ \\
2) For all teams $X$ such that $dom(X)\cap Z = \emptyset$, and for any structure $M$, $M,X\models \psi \Leftrightarrow M,X\models \chi$.\footnote{Actually, the notion of $Z$-equivalence in \cite{CaiDecJan2009} also contains a requirement about negative satisfiability. Since we have not introduced negative satisfiability in this paper, we will ignore this aspect.}
\end{quote} 
The case with $Z=\emptyset$ is the usual notion of truth-equivalence.


\subsection*{Prenex form theorem}

We move towards a prenex form result. First of all, we need a rule for the extraction of quantifiers in $IF(\mathcal R)$. We can reuse the proof scheme of the analogous result for $IF^*$ from \cite{CaiDecJan2009}; but since that proof is based on many intermediate results, we need to check carefully that the lemmas generalize to our case.

\begin{lem} \label{EXPSIM}
 \citep[Lemma 7.4]{CaiDecJan2009} Let $V$ be a set of variables, and $v$ a variable not in $V$, $s,t$ assignments of domain $V$ and codomain $M$, and $W \subseteq V$. Then
\[
    s\sim_{W}t \Leftrightarrow s(a/v)\sim_{Wv}t(b/v)
\]
for all $a,b\in M$.
\end{lem}

\begin{lem}\label{SED}
Let $\varphi$ be an $IF^*(\mathcal R)$ formula, $X$ a team and $v$ a variable not occurring in $\varphi$ nor in $dom(X)$. Then, for any $v$-expansion $X_v$ of $X$,
\[
M,X\models\varphi \Longleftrightarrow M,X_v\models \varphi_{/v}.
\]
\end{lem}

\begin{proof}
This can be proved by induction on the structure of $\varphi$. See \cite{Dec2005}, 5.5 for an exhaustive treatment of cases, including slashed disjunctions. The proof of the existential case can be taken as a model for the intermediate quantifier case.
\end{proof}

The following is the extraction rule we need. Remember that we use the letter $R$ to denote quantifiers distinct from $\forall$ and $\exists$ (``intermediate'' quantifiers).

\begin{lem}\label{WEAKEXT}
For any formulas $\psi,\chi$ of $IF^*(R)$, any variable $v$ not occurring in $\chi$, $V$ nor $W$, and $Q$ being either $\forall,\exists$ or an intermediate $R$,
\[
(Qv/V)\psi \lor_{/W} \chi \equiv_v (Qv/V)(\psi \lor_{/Wv} \chi_{/v})
\]
\end{lem}

\begin{proof}
For the cases $Q=\forall,\exists$ see the proof of Theorem 7.5 in \cite{CaiDecJan2009}. We examine the case $Q=R$. The requirements that $v\notin V\cup W$ and that $v$ does not occur in $\chi$ ensure that the non-triviality condition for $v$-equivalence is respected. 

$\Rightarrow$) This part of the  proof does not differ significantly from the corresponding existential case in \cite{CaiDecJan2009}, Theorem 7.5, so we omit it.


$\Leftarrow$) Suppose $M,X\models (Rv/V)(\psi \lor_{/Wv} \chi_{/v})$ for some team $X$ such that $v\notin dom(X)$. Then $M,Y_1\models\psi$ and $M,Y_2\models \chi_{/v}$, where the $Y_i$ are $Wv$-uniform and  $Y_1\cup Y_2 = X[F/v]$ for some $V$-uniform function $F:X\rightarrow R^M$. Now let $s\in Y_i, s'\in X[F/v]$ and $s\sim_v s'$. This obviously implies $s\sim_{Wv} s'$. So, by $Wv$-uniformity of $Y_i$, $s'\in Y_i$. Consequently, $Y_i=X_i[F/v]$ for some $X_i\subseteq X$.
Thus $M,X_1[F/v]\models \psi$ (from which it follows that $M,X_1\models (Rv/V)\psi$) and, thanks to Lemma \ref{SED} (which is applicable since $v\notin dom(X)$), $M,X_2\models\chi$. Clearly $X_1\cup X_2=X$, otherwise the $Y_i$ would not cover $X[F/v]$. We check that the $X_i$ are $W$-uniform. Suppose $s\in X_i, t\in X,s\sim_{W}t$. Then, by lemma \ref{EXPSIM},  $s(F(s)/v)\sim_{Wv}t(F(t)/v)$; so, by $Wv$-uniformity of $Y_i$, $t(F(t)/v)\in Y_i$. So, $t\in X_i$. We may conclude that $M,X\models(Rv/V)\psi \lor_{/W} \chi$. 
\end{proof}

The following basic result holds as usual:

\begin{lem}[Interchanging free variables]        \label{INTERCHANGING}
If $x\notin Bound(\psi)$ and $z$ does not occur in $\psi$, then for any structure $M$ and any suitable team $X$ such that $x\in dom(X)$ and $z\notin dom(X)$,
\[
M,X\models \psi \Longleftrightarrow M, X_{[z/x]} \models \psi[z/x].
\]

\end{lem}

\begin{thm} \label{RENAMEVAR}
Let $z$ be a variable not occurring in $(Qx/X)\psi$, where $Q$ is either $\forall,\exists$ or $R_i$. Then:\\
a) If $x$ does not occur bound in $\psi$ and does not occur in $X$, then
\[
     (Qx/X)\psi \equiv_{xz} (Qz/X)\psi[z/x]
\]
b)  If $x$ does not occur bound in $\psi$, then
\[
     (Qx/X)\psi \equiv_{z} (Qz/X)(\psi[z/x]_{/\{x\}})
\]
\end{thm}

\begin{proof}
a-b) The proof of \cite{ManSanSev2011}, Theorem 5.37, applies almost without changes.\\

\end{proof}

The notions of \emph{regular}\footnote{Not to be confused with the homonym notion from \cite{ManSanSev2011}.} and \emph{strongly regular} formula from \cite{CaiDecJan2009} are also sensible in our context.

\begin{definition} An $IF^*(\mathcal R)$ formula $\psi$ is \textbf{regular} if:\\
1) No variable occurs both bound and free in $\psi$\\
2) No quantifier over a variable, say $v$, occurs in the scope of another quantifier over $v$.

An $IF^*(\mathcal R)$ formula $\psi$ is \textbf{strongly regular} if each variable is quantified at most once in it.
 
\end{definition}

\begin{lem} \label{REGULARSUBST}
Let $\varphi$ and $\varphi(\chi/\psi)$ be regular $IF^*(\mathcal R)$ formulas. If $\psi \equiv_V \chi$, then $\varphi \equiv_V \varphi(\chi/\psi)$
\end{lem}

\begin{proof}
One can use the proof of Theorem 9.6 in \cite{CaiDecJan2009}, checking that the lemmas it depends on (6.14, 6.16, and our \ref{INTERCHANGING}) and their proofs are also valid for $IF^*(\mathcal R)$.
\end{proof}

\begin{thm}[Strong regularization]   \label{STRONGREG}
Every $IF(\mathcal R)$ (resp. $IF^*(\mathcal R)$) formula $\psi$ is $V$-equivalent to a strongly regular
$IF(\mathcal R)$ (resp. $IF^*(\mathcal R)$) formula $\psi'$, for some set of variables $V\subseteq Bound(\psi')$; $\psi'$ can be chosen so that $Bound(\psi)\cap Bound(\psi')= \emptyset$.
\end{thm}

\begin{proof}
The proofs of 9.3 and 9.4 from \cite{CaiDecJan2009} still work, using Theorem \ref{RENAMEVAR} and \ref{REGULARSUBST}
\end{proof}

\begin{thm}[Prenex normal form theorem]  Every $IF(\mathcal R)$ (resp. $IF^*(\mathcal R)$) formula is $\emptyset$-equivalent
to a strongly regular $IF(\mathcal R)$ (resp. $IF^*(\mathcal R)$) formula in prenex normal form.
\end{thm}
\begin{proof}
It can be proven as in \cite{CaiDecJan2009}, 10.1, finding first a strongly regular equivalent (Theorem \ref{STRONGREG}) and then using renaming (Theorem \ref{RENAMEVAR}) and quantifier extraction (Theorem \ref{WEAKEXT}).
\end{proof}

\subsection*{Strong extraction rule}

The prenex normal form theorem does not exhaust the discourse about prenex transformations. As observed in \cite{CaiDecJan2009}, the (weak) extraction rule that we gave above is somewhat unsatisfactory, in that it does not yield as a special case the quantifier extraction rule of first order logic. We want to show that a strong extraction rule -- analogous to \cite{CaiDecJan2009}, Theorem 8.3 -- can still be obtained for Engstr\"om quantifiers.

To make it work in the $IF^*(\mathcal R)$ case, we need some tools for eliminating variables from singleton slash sets of slashed disjunctions. The following somewhat unintuitive rule is similar (but less general) to an analogous (and as unintuitive) result for $\forall$, Lemma 8.2 of \cite{CaiDecJan2009}; the proof is a bit more involved.

\begin{lem} \label{RIRR}
Let $\psi,\chi$ be $IF^*(\mathcal R)$ formulas, and $v$ a variable not occurring in $\chi$ nor $V$. Then, for any intermediate quantifier $R\in\mathcal R$,
\[
(Rv/V)(\psi\lor_{/v}\chi_{/v})\equiv_v (Rv/V)(\psi\lor\chi_{/v}).
\]
\end{lem}

\begin{proof}
The condition $v\notin V$ ensures the non-triviality of the equivalence relation.

For the left-to-right implication, just notice that any $v$-uniform partition of a team is also a  partition (without further specifications).

From right to left. We have $M,X[F/v]\models(\psi\lor\chi_{/v})$, where $F$ is a $V$-uniform function $X\rightarrow R^M$ and $v\notin dom(X)$. Then there are teams $Y_1,Y_2$ such that $X[F/v]=Y_1\cup Y_2$ and $M,Y_1\models\psi,M,Y_2\models\chi_{/v}$. By Lemma \ref{SED}, $M,(Y_2)_{-v}\models \chi$. Notice that, since $v\notin dom(X)$, $(Y_2)_{-v}\subseteq X$. So we can expand the domain again using \ref{SED}: if $F'$ is the restriction of $F$ to $(Y_2)_{-v}$, we have $M,(Y_2)_{-v}[F'/v]\models\chi_{/v}$. Clearly $Y_2\subseteq (Y_2)_{-v}[F'/v]\subseteq X[F/v]$. We want to show that $(Y_2)_{-v}[F'/v]$ is $v$-uniform in $X[F/v]$. Let $s(a/v) \sim_{v} t(b/v)$, where $s\in (Y_2)_{-v}, t\in X, a\in F'(s), b\in F(t)$. Since $v\notin dom(s)=dom(t)$, then $s(a/v) \sim_{v} t(b/v)$ implies $s=t$ (this is an almost trivial case of lemma \ref{EXPSIM}). So, $b \in F(s)$. Thus $t(b/v)=s(b/v)\in (Y_2)_{-v}[F'/v]$, and we have proved that $(Y_2)_{-v}[F'/v]$ is $v$-uniform. Thus, also the complement $X[F/v]\setminus (Y_2)_{-v}[F'/v]$ is $v$-uniform. Since $X[F/v]\setminus (Y_2)_{-v}[F'/v]\subseteq Y_1$, by downward closure (\ref{DWCR}) $M,X[F/v]\setminus (Y_2)_{-v}[F'/v]\models\psi$. So, $M,X[F/v]\models\psi\lor_{/v}\chi_{/v}$, and thus $M,X\models (Rv/V)(\psi\lor_{/v}\chi_{/v})$.
 \end{proof}

For a comparison, the above-mentioned result 8.2 of \cite{CaiDecJan2009} allows transforming a disjunction of the form $\lor_{/Wv}$, occurring immediately below $\forall x$, into $\lor_{/W}$ (in our version for $R$, $W$ must be empty).  For $\exists$ there is a rather different property which allows the transformation of $\lor_{/Wv}$ into $\lor_{/W}$ below an existential quantifier $(\exists x/V)$: it works under the hypotheses that $W\subseteq V$ and $v\notin V$ (8.1 of \cite{CaiDecJan2009}). \\
\\
In the following we use the notation $\psi|_v$ to denote the formula which is obtained from $\psi$ when one adds variable $v$ to all the \emph{nonempty} slash sets.

\begin{lem} \label{VERTICALIZE}
 Let $\psi$ be an $IF^*(\mathcal R)$ formula without occurrences of variable $v$. Then 
\[
\psi_{/v} \equiv \psi|_v.
\]  
\end{lem}

\begin{proof}
From left to right, this can be easily proved by induction on the structure of a formula (also thanks to the fact that the subformulas of $\psi$ have no occurrence of $v$).

Also the right-to-left implication can be proved by structural induction. Most cases have been taken care of in the proof of Theorem 6.7 of \cite{CaiDecJan2009}. We just prove the case $\psi = R_u\chi$, which is new, and the case  $\psi = \exists u\chi$, which we believe has not been treated correctly in \cite{CaiDecJan2009}.\footnote{The slashed cases $(Qy/Y)\chi_{|x}$ are uninteresting, because by definition, if $Y$ is nonempty, $(Qy/Y)\chi_{|x})_{|x} = (Qy/Y\cup\{x\})\chi_{|x}$, which by induction is equal to $(Qy/Y\cup\{x\})\chi_{/x} = ((Qy/Y)\chi)_{/x}$).} 

We begin with the latter. So, $\psi|_v$ is $\exists u\chi|_v$. If $M,X\models\psi|_v$ (where $X$ is a team whose domain contains $FV(\psi|_v)\cup\{v\}$) then there is a function $F:X\rightarrow \exists^M$ such that $M,X[F/u]\models\chi|_v$. By induction hypothesis, $M,X[F/u]\models\chi_{/v}$. By Lemma \ref{SED}, $M,X[F/u]_{-v}\models\chi$.  Now define a function $G:X_{-v}\rightarrow \exists^M$ by $G(s)=F(s(a_s/v))$, where $a_s$ is a chosen element of $M$ such that $s(a_s/v)\in X$. Then $X_{-v}[G/u]\subseteq X[F/u]_{-v}$. By downward closure (\ref{DWCR}), $M,X_{-v}[G/u]\models\chi$. So, $M,X_{-v}\models \exists u\chi$, and using Lemma \ref{SED} again, $M,X\models (\exists u\chi)_{/v}$. 

$R$ can be treated analogously, constructing a function $G:X_{-v}\rightarrow R^M$ from a function $F:X\rightarrow R^M$.
Just notice that, in case $F(s(a_s/v))\neq\emptyset$ for some $s$, we can define $G(s) := F(s(a_s/v))$; and in case no such value exists, define $G(s):=\emptyset$. In the special case that $F=\emptyset$, we can just set $G:=\emptyset$; the rest of the arguments holds because all of the teams involved have the same domain (the empty set of variables).
\end{proof}

The following is the strong extraction rule:

\begin{thm} \label{STRONGEXT}
Let $\mathcal R$ be a set of intermediate quantifiers. Let $Q$ be either $\exists,\forall$ or $R_i$. Let $\psi$, $\chi$ be $IF^*(\mathcal R)$ formulas, with $v$ not occurring in $\chi$ nor in $V$. Then
\[
(Qv/V)\psi\lor\chi \equiv_v (Qv/V)(\psi\lor\chi|_v). 
\]
\end{thm}
\begin{proof}
The assertion $v\notin V$ guarantees that the nontriviality condition for $v$-equivalence is respected.

By the weak extraction rule \ref{WEAKEXT}, we have 
\[
(Qv/V)\psi\lor\chi \equiv_v (Qv/V)(\psi\lor_{/\{v\}}\chi_{/v}). 
\]
Use the lemmas for the elimination of slashes under quantification -- Lemma 8.1 of \cite{CaiDecJan2009} for $\exists$; Lemma 8.2 for $\forall$; our Lemma \ref{RIRR} for $R \in \mathcal R$ --  to transform the rightmost formula into
\[
(Qv/V)(\psi\lor\chi_{/v}).
\]
Finally, by Lemma \ref{VERTICALIZE} and substitution of $\emptyset$-equivalent formulas \citep[Theorem 6.5.3]{CaiDecJan2009}, the rightmost formula becomes (up to $\emptyset$-equivalence):
\[
(Qv/V)(\psi\lor\chi|_v).
\]
\end{proof}

From this, the classical extraction rule immediately follows:

\begin{cor}
Let $\mathcal R$ be a set of intermediate quantifiers. Let $Q$ be either $\exists,\forall$ or $R_i$. Let $\psi$, $\chi$ be $FO(\mathcal R)$ formulas (with empty slash sets), with $v$ not occurring in $\chi$. Then
\[
Qv\psi\lor\chi \equiv_\emptyset Qv(\psi\lor\chi). 
\]
\end{cor}

\begin{proof}
Since $\chi$ is in $FO(\mathcal R)$, $\chi|_v=\chi$. So, by \ref{STRONGEXT}, $Qv\psi\lor\chi \equiv_v Qv(\psi\lor\chi)$. Since $v$ is not a free variable of $Qv\psi\lor\chi$ nor $Qv(\psi\lor\chi)$, and $FO(\mathcal R)$ is a local logic, the thesis follows.
\end{proof}

\subsection*{Primality test}

Next we turn to the primality test \citep[Theorem 18]{Sev2014}, which is a rather general syntactical criterion for understanding whether a regular and prenex $IF(\mathcal R)$ sentence is equivalent to some (unslashed) $FO(\mathcal R)$ sentence. These kinds of criteria, even though they are just sufficient and not necessary, are of interest, because it is known that the problem is undecidable already for 
 $IF$ logic. There is no need to enter here into the details of the criterion; suffice it to say that it was proven for the ``most'' quantifier, in \cite{Sev2014}, using only three equivalence rules:\\
1) Swapping mutually independent quantifiers\\
2) Making slash sets empty, whenever they only contain existentially quantified variables\\
3) Making the slash sets of universal quantifiers empty\footnote{Always possible when only truth, and not falsity is considered; this has been our approach in the whole paper.}.\\

So, we just need to prove that these rules hold in $IF(\mathcal R)$ to obtain the primality criterion. We do not manage to prove 1), and thus primality, for \emph{all} quantifiers: we have to add an (almost trivial) restriction.

\begin{definition}
A quantifier $Q$ is \textbf{emptyset-free} if, for every structure $M$, $\emptyset\notin Q^M$.
\end{definition}

\begin{thm}[Swapping independent quantifiers] \label{SWAPR}
Let $\mathcal R$ be intermediate quantifiers, and $R_i,R_j\in \mathcal  R$ emptyset-free. Suppose $u,v$ are two distinct variables. Then, for any $IF(\mathcal R)$ formula $\psi$:\\
a) $(R_iu/U)(R_jv/Vu)\psi\equiv_{uv}(R_jv/V)(R_iu/Uv)\psi$\\
b) $(\exists v/V)(R_ju/Uv)\psi\equiv_{uv}(R_ju/U)(\exists v/Vu)\psi$.\\
In case $R_i$ is not emptyset-free, we still have (with self-explaining notations):\\
a') $(R_iv/V)(R_ju/Uv)\psi \models_{uv} R_ju(R_iv/Vu)$\\
a'') $(R_ju/U)(R_iv/Vu)\psi\models_{uv}(R_iv/V)(R_ju/Uv)\psi$\\
b') $(R_ju/U)(\exists v/Vu)\psi\models_{uv}(\exists v/V)(R_ju/Uv)\psi$\\
b'') $(\exists v/V)(R_ju/Uv)\psi   \models_{uv}   R_ju(\exists v/Vu)\psi$
\end{thm}

\begin{proof}
a,a',a'') From left to right. Suppose $M,X\models (R_iu/U)(R_jv/Vu)\psi$, with $u,v\notin dom(X)$. Then there are a $U$-uniform function $F:X\rightarrow R_i^M$ and a $Vu$-uniform function $G:X[F/u]\rightarrow R_j^M$ such that $M,X[F,G/u,v]\models\psi$. 

Define a function $G^*:X\rightarrow R_j^M$ by 
\[
G^*(s) = \bigcup\{G(s(a/u)) | a \in F(s)\}. 
\] 

Since $G$ is $Vu$-uniform, $G(s(a/u)) = G(s(b/u))$ for any $a,b\in M$. So in reality $G^*(s) = G(s(a/u))\in R_j^M$ for any $a \in F(s)$; or, in the special case that $G=\emptyset$, $G^*=\emptyset$. 
In case $\emptyset\notin R_j^M$, one can also prove that $G^*$ is a $V$-uniform function. Indeed, suppose $s\sim_V s'$. Since $G$ is $Vu$-uniform then $G(s(a/u)) = G(s'(b/u))$ for all $a,b\in M$; and since $\emptyset\notin R_i^M$, there exist $a\in F(s)$ and $b\in F(s')$. Thus $G^*(s) = G(s(a/u)) = G(s'(b/u)) = G^*(s')$, where $a\in F(s)$ and $b\in F(s')$.

Define then a function $F^*: X[G^*/v]\rightarrow R_i^M$ by $F^*(s) = F(s_{-v})$. This is well defined because $v \notin dom(X)$, so that $s_{-v}\in X$. Suppose $s\sim_{Uv}s'$. Then $s_{-v}\sim_U s'_{-v}$, so by the $U$-uniformity of $F$, $F^*(s) = F(s_{-v})=F(s'_{-v}) =F^*(s')$. Thus $F^*$ is $Uv$-uniform. One may check that $X[G^*,F^*/v,u] = X[F,G/u,v]$. So $M,X\models (R_jv/V)(R_iu/Uv)\psi$. From right to left one may use the same argument.\\
b,b',b'') The same proof method can be applied.
\end{proof}

This result is in contrast with the counterexample given in \cite{Eng2012}, sect.2.2, that falsifies the equivalence $M,X\models^b\exists_{=1} x (\exists y/x)\psi\Leftrightarrow M,X\models^b\exists y(\exists_{=1} x/y)\psi$ and shows the failure (even for emptyset-free quantifiers) of quantifier swapping under Engstr\"om's \emph{second} semantical clause.


\begin{lem} \label{EXISTENTIALDEPR}
Suppose in the $IF^*(\mathcal R)$ formula $\varphi$ there is an occurrence of a quantifier $(Qv/V)$, where $V$ contains only existentially quantified variables. Let $\varphi'$ be obtained from $\varphi$ by replacing $V$ with the empty set. Then $\varphi\equiv\varphi'$.
\end{lem} 

\begin{proof}
We can observe that the proof of the analogous result for prefixes (of $IF$ logic plus the ``most'' quantifier) given in \cite{Sev2014}, Lemma 12, still works.
One only has to add the observation that the act of splitting teams when disjunctions are analysed preserves the property that each assignment in the team is uniquely determined by the values it assigns to universally and $R$-quantified variables.
\end{proof}

A last, obvious generalization is the following:

\begin{lem} \label{UNIVERSLEMMAR}
Every $IF^*(\mathcal R)$ formula of the form $(\forall v/V)\psi$ is truth-equivalent to $\forall v\psi$. Every $IF^*(\mathcal R)$ formula of the form $(\psi_1 \land_{/W} \psi_2)$ is truth equivalent to $(\psi_1 \land \psi_2)$.
\end{lem}

The three last results, when restricted to $IF(\mathcal R)$, justify the primality test for such logics.\\
\\

\bibliography{qfbib}{}

\vspace*{10pt}
\end{document}